\documentclass[12 pt]{article}
\usepackage[pdftex]{graphicx}
\usepackage{amsmath}
\usepackage{mathtools}
\usepackage{amsthm}
\usepackage{amsfonts}
\usepackage{amssymb}
\usepackage[margin=1.0in]{geometry}
\usepackage{tikz}

\newtheorem{theorem}{Theorem}[section]
\newtheorem{cor}[theorem]{Corollary}
\newtheorem{lemma}[theorem]{Lemma}
\newtheorem{prop}[theorem]{Proposition}

\newtheorem{prob}[theorem]{Problem}
\newtheorem{conj}[theorem]{Conjecture}

\theoremstyle{definition}

\theoremstyle{remark}
\newtheorem*{rem}{Remark}
\newtheorem*{claim}{Claim}

\newcommand{\flim}[1]{\mathrm{Flim}(#1)}
\newcommand{\age}[1]{\mathrm{Age}(#1)}
\newcommand{\fin}[1]{\mathrm{Fin}(#1)}
\newcommand{\fr}{Fra\"iss\'e }
\renewcommand{\phi}{\varphi}
\newcommand{\arr}[1]{\mathrm{Arr}(#1)}
\newcommand{\ob}[1]{\mathrm{Ob}(#1)}

\newcommand{\emb}[1]{\mathrm{Emb}(#1)}
\newcommand{\aut}[1]{\mathrm{Aut}(#1)}

\newcommand{\dom}[1]{\mathrm{dom}(#1)}

\newcommand{\im}[1]{\mathrm{Im}(#1)}

\begin{document}
\title{Topological Dynamics of Closed Subgroups of $S_\infty$}
\author{Andy Zucker}
\date{April 2014}
\maketitle

\begin{abstract}
For $G$ a closed subgroup of $S_{\infty}$, we provide an explicit characterization of the greatest $G$-ambit. Using this, we provide a precise characterization of when $G$ has metrizable universal minimal flow. In particular, each such instance fits into the framework of metrizable flows developed in [KPT] and [NVT]; as a consequence, each $G$ with metrizable universal minimal flow has the generic point property.
\let\thefootnote\relax
\let\thefootnote\relax
\end{abstract}

\section{Introduction}

In the study of abstract topological dynamics, one is often concerned with the continuous action of a Hausdorff topological group $G$ on a compact Hausdorff space $X$, often called a \emph{$G$-flow}. The flow $X$ is \emph{minimal} if every orbit is dense and \emph{universal} if for every $G$-flow $Y$, there is a $G$-map $f: X\rightarrow Y$, where a \emph{$G$-map} is a continuous map which respects the $G$-action. It is a fact that every topological group $G$ admits a \emph{universal minimal flow} $M(G)$ which is unique up to $G$-flow isomorphism. 

One common tool used to study the universal minimal flow $M(G)$ is the \emph{greatest ambit}. A \emph{$G$-ambit} $(X, x_0)$ is a $G$-flow $X$ with a distinguished point $x_0\in X$ whose orbit is dense in $X$. The greatest ambit is then an ambit which maps onto every other $G$-ambit, where a map of $G$-ambits is a $G$-map which also respects the distinguished point. Since any minimal $G$-flow can be turned into an ambit by distinguishing any point, it follows that every minimal subflow of the greatest ambit is universal, hence isomorphic to $M(G)$.

A major direction of research for the past two decades has been the attempt to classify those Polish groups $G$ for which $M(G)$ is metrizable. The introduction of the seminal paper by Kechris, Pestov and Todor\v{c}evi\'c [KPT] contains an excellent survey of early efforts in this direction. In this paper, the authors provide a general way of constructing $M(G)$ for many closed subgroups of $S_\infty$, the group of permutations of $\mathbb{N}$ endowed with the pointwise convergence topology. Interestingly, the greatest ambit is not the primary tool used to study $M(G)$ in this case.

The closed subgroups of $S_\infty$ are exactly those Polish groups which are \emph{non-Archimedean}, i.e.\ which admit a neighborhood basis at the identity consisting of open subgroups. The characterization we will find most useful is that the closed subgroups of $S_\infty$ are exactly the automorphism groups of countably infinite, model-theoretic structures with universe $\mathbb{N}$. We can in fact narrow our scope to certain countably infinite structures known as \emph{\fr structures}. These are those countably infinite structures $\mathbf{K}$ with universe $\mathbb{N}$ which are:
\begin{itemize}
\item
\emph{locally finite} -- there are finite substructures $\mathbf{A}_n\subseteq \mathbf{K}$ with $\mathbf{K} = \bigcup_n \mathbf{A}_n$, 
\item
\emph{ultrahomogeneous} -- every isomorphism $f: \mathbf{A}\rightarrow \mathbf{B}$ between finite substructures of $\mathbf{K}$ extends to an automorphism of $\mathbf{K}$.
\end{itemize}
Examples of \fr structures include the countably infinite set, the rational linear ordering, the random graph, and the countable atomless boolean algebra. See [KPT] for many more examples.

The most useful aspect of \fr structures is that they are uniquely determined by their \emph{age}, the class of finite structures which embed into $\mathbf{K}$. The major insight of [KPT] is that the dynamical properties of $\aut{\mathbf{K}}$ can be studied using the combinatorial properties of $\age{\mathbf{K}}$. Of particular importance is the notion of (structural) \emph{Ramsey degree}: 
\begin{itemize}
\item
If $\mathbf{A}, \mathbf{B}$ are finite substructures, let $\binom{\mathbf{B}}{\mathbf{A}}$ denote the set of substructures of $\mathbf{B}$ which are isomorphic to $\mathbf{A}$. Let $\mathcal{K}$ be a class of finite structures, and for $n\in \mathbb{N}$, set $[n] = \{1,2,...,n\}$. We say that $\mathbf{A}\in\mathcal{K}$ has (structural) \emph{Ramsey degree} $\leq k$ if for every $\mathbf{B}\in\mathcal{K}$ with $\binom{\mathbf{B}}{\mathbf{A}}$ nonempty and every $r\in\mathbb{N}$, there is $\mathbf{C}$ in $\mathcal{K}$ such that for every coloring $\gamma: \binom{\mathbf{C}}{\mathbf{A}}\rightarrow [r]$, there is $\mathbf{B}_0\in \binom{\mathbf{C}}{\mathbf{B}}$ with $|\gamma(\binom{\mathbf{B_0}}{\mathbf{A}})| \leq k$.
\end{itemize}

If $\mathbf{A}$ has Ramsey degree $1$, we say that $\mathbf{A}$ is a (structural) \emph{Ramsey object}. We say that $\mathcal{K}$ has the (structural) \emph{Ramsey Property} if every $\mathbf{A}\in\mathcal{K}$ is a Ramsey object. In section 4, we will introduce the (embedding) Ramsey Property, and most of this paper will use this rather than the structural version above. For now, we note the following for $\mathcal{K}$ a class of finite structures:

\begin{itemize}
\item
$\mathcal{K}$ has the (embedding) \emph{Ramsey Property} iff $\mathcal{K}$ has the (structural) Ramsey Property and consists of \emph{rigid} structures, i.e.\ structures with no non-trivial automorphisms. 
\item
$\mathbf{A}\in\mathcal{K}$ has finite (structural) Ramsey degree iff $\mathbf{A}$ has finite (embedding) Ramsey degree.
\end{itemize}
We can now state the first major theorem in [KPT].

\begin{theorem}
Let $\mathbf{K}$ be a \fr structure with $\age{\mathbf{K}} = \mathcal{K}$. Set $G = \aut{\mathbf{K}}$. Then the universal minimal flow $M(G)$ is a single point iff $\mathcal{K}$ has the embedding Ramsey Property.
\end{theorem}

Topological groups $G$ with $M(G)$ a single point are called \emph{extremely amenable}. Another major theme of [KPT] is that if $\mathbf{K}$ is a \fr structure with universe $\mathbb{N}$ and $G = \aut{\mathbf{K}}$ is not extremely amenable, we can often express $M(G)$ as a \emph{logic action}. Let $X_{LO}^\mathbf{K}$ be the space of structures of the form $\langle \mathbf{K}, <\rangle$, where $<$ is a linear ordering of $\mathbb{N}$. We endow $X_{LO}^\mathbf{K}$ with the topology whose basic open neighborhoods are of the form $N(L) = \{\langle \mathbf{K},<\rangle\in X_{LO}^\mathbf{K}:\: <\!\!|_{[k]} = L\}$, where $L$ is some linear ordering of $[k] = \{1,2,...,k\}$. With this topology, $X_{LO}^\mathbf{K}$ is compact and metrizable. We let $G$ act on $X_{LO}^\mathbf{K}$ via $\langle \mathbf{K}, <\rangle\cdot g = \langle \mathbf{K}, <_g\rangle$ where $(m <_g n)$ iff $(g(m) < g(n))$. This turns $X_{LO}^\mathbf{K}$ into a $G$-flow. It should be noted that while [KPT] uses left logic actions, this paper will almost always use right logic actions; this choice will be justified in section 6. When $\mathbf{K}$ is a \fr structure, $\langle \mathbf{K}, <_0\rangle$ is also a \fr structure, and $G = \aut{\mathbf{K}}$, [KPT] provides a complete characterization of when $\overline{\langle \mathbf{K}, <_0\rangle \cdot G} \cong M(G)$. 

Nguyen Van Th\'e in [NVT] made the observation that by allowing more general logic actions, one can describe $M(G)$ for more groups. Let $\mathcal{S} = \{S_i\}_{i\in I}$ be a set with countably many new relation symbols of arity $n(i)$. Let $X_{\mathcal{S}}^\mathbf{K}$ be the space of structures of the form $\langle \mathbf{K}, \{S_i^\mathbf{K}: i\in I\}\rangle$. We endow $X_\mathcal{S}^\mathbf{K}$ with the topology whose basic open neighborhoods are the sets $N(\{T_i: i\in I\}) := \{\langle \mathbf{K},\{S_i^\mathbf{K}: i\in I\}\rangle \in X_\mathcal{S}^\mathbf{K}: S_i^\mathbf{K}|_{[k]} = T_i\}$, where $T_i$ is some interpretation of the relation symbol $S_i$ on $[k]$. This topology is metrizable using a similar metric to the linear order case. The right logic action is also defined similarly; if $\mathbf{K}$ is a \fr structure, $\langle \mathbf{K}, \vec{S}^\mathbf{K}\rangle \in X_\mathcal{S}^\mathbf{K}$ is also a \fr structure, and $G = \aut{\mathbf{K}}$, [NVT] provides a complete characterization of when $\overline{\langle \mathbf{K}, \vec{S}^\mathbf{K}\rangle \cdot G}$ is compact and isomorphic to $M(G)$.

One interesting consequence of these characterizations is that if $M(G) = \overline{\langle \mathbf{K}, \vec{S}^\mathbf{K}\rangle\cdot G}$ as above, then $M(G)$ has a dense $G_\delta$ orbit, namely $\langle \mathbf{K}, \vec{S}^\mathbf{K}\rangle\cdot G$. If $G$ is a topological group and $M(G)$ has a generic orbit, then we say $G$ has the \emph{generic point property}. It can be shown (see Prop.\ 5.10) that if $M(G)$ has a generic orbit, then every minimal $G$-flow has a generic orbit. Angel, Kechris, and Lyons in [AKL] conjecture that each Polish group $G$ with $M(G)$ metrizable has the generic point property; this conjecture is known as the \emph{Generic Point Problem}. For Polish groups $G$ with the generic point property and $M(G)$ metrizable, Melleray, Nguyen Van Th\'e, and Tsankov in [MNT] have shown that $M(G)$ must have a very particular structure; see section 9 for a very brief discussion.

In this paper, we first provide an explicit characterization of the greatest $G$-ambit for $G$ a closed subgroup of $S_\infty$. Using this, we provide new proofs of the main theorems of [KPT] and [NVT]. Furthermore, our new proof will show that all closed subgroups of $S_\infty$ with $M(G)$ metrizable fit into the framework of logic actions developed in [NVT]. As a consequence, we settle in the affirmative the Generic Point Problem for closed subgroups of $S_\infty$. Our main theorem is the following:

\begin{theorem}
Let $\mathcal{K}$ be a \fr class, $\mathbf{K} = \flim{\mathcal{K}}$, and $G = \aut{\mathbf{K}}$. Then the following are equivalent:
\begin{enumerate}
\item
$G$ has metrizable universal minimal flow,
\item 
Each $\mathbf{A}\in \mathcal{K}$ has finite Ramsey degree,
\item
There is a countable set $\mathcal{S}$ of new relation symbols and $\langle \mathbf{K}, \vec{S}^\mathbf{K}\rangle\in X_{\mathcal{S}}^\mathbf{K}$ so that $\langle \mathbf{K}, \vec{S}^\mathbf{K}\rangle$ is also a \fr structure and $M(G)\cong \overline{\langle \mathbf{K},\vec{S}^\mathbf{K}\rangle\cdot {G}}$.
\end{enumerate}
\end{theorem}

The paper is organized as follows. Sections 2 through 5 provide a review of topology, \fr structures, structural Ramsey theory, and KPT correspondence, respectively. Section 6 provides an explicit construction of the greatest ambit for closed subgroups of $S_\infty$, and sections 7 and 8 give new proofs of the main theorems in [KPT]. As a warning, sections 3,4, and 5, while mostly review, do contain some new notions. Section 3 introduces the notion of a \emph{Fra\"iss\'e--HP} class (read ``\fr minus HP''), and section 5 discusses precompact expansions on Fra\"iss\'e--HP classes. Section 4 introduces the notions of (embedding) Ramsey Property/degree/object and contains some other new ideas and nonstandard vocabulary.

\subsection*{Acknowledgements}

I would like to thank James Cummings, Alekos Kechris, Julien Melleray, Miodrag Soki\'c, Lionel Nguyen Van Th\'e, and Todor Tsankov for their many useful comments and suggestions. Their help has been invaluable.

\section{Topological Preliminaries}

In this section, we will discuss the topological tools needed going forward. We should note now that \emph{all} topological spaces  and groups are assumed to be Hausdorff unless explicitly stated otherwise; in particular, any results stated for a class of topological spaces should only be presumed to hold for those members of the class which are Hausdorff.

\subsection{Topological Dynamics and Topological Semigroups}

Let $G$ be a topological group. A (right) \emph{$G$-flow} is a pair $(X, \tau)$, where $X$ is a compact space and $\tau: X\times G\rightarrow G$ is a continuous action, i.e.\ for every $x\in X$ and $g,h\in G$, we have $\tau(\tau(x,g),h) = \tau(x, gh)$. Typically the action $\tau$ is understood and suppressed, so we write $x\cdot g$ for $\tau(x,g)$, or simply $xg$ when there is no confusion. Then we have the identity $x\cdot (gh) = (x\cdot g)\cdot h$. A \emph{subflow} of $X$ is a non-empty closed subspace $Y\subseteq X$ for which $y\cdot g\in Y$ for all $y\in Y$ and $g\in G$. As $X$ is compact, we see that the intersection of a decreasing chain of subflows of $X$ is itself a subflow. Applying Zorn's lemma, we see that $X$ contains a \emph{minimal} subflow $Y$, a flow containing no proper subflows. Notice that if $Y$ is minimal and $y\in Y$, then the \emph{orbit closure} $\overline{y\cdot G}$ is a subflow of $Y$, so we must have $\overline{y\cdot G} = Y$. More generally, a flow $Y$ is minimal iff every orbit is dense.

If $X$ and $Y$ are $G$-flows, a $G$-map $f: X\rightarrow Y$ is a continuous map which respects the $G$-action, i.e.\ $f(x\cdot g) = f(x)\cdot g$ for each $x\in X$ and $g\in G$. Notice that the dots on the left and the right express different $G$-actions. An \emph{isomorphism} of $G$-flows is a bijective $G$-map (by compactness, the inverse is continuous, hence also a $G$-map). A flow $X$ is \emph{universal} iff for each minimal flow $Y$, there is a $G$-map $f: X\rightarrow Y$. It is a fact that every  topological group $G$ admits a unique \emph{universal minimal flow} $M(G)$ up to $G$-flow isomorphism. The rest of this section will be spend proving this fact. The proof we will use is to first prove the existence and uniqueness of the greatest $G$-ambit $S(G)$. Then any minimal subflow of $S(G)$ is universal, and we will show that any universal minimal flow is isomorphic to any minimal flow of $S(G)$.  

A \emph{$G$-ambit} $(X,x_0)$ consists of a $G$-flow $X$ and a distinguished point $x_0\in X$ with dense orbit. A typical example of a $G$-ambit is the \emph{orbit closure}: start with any $G$-flow $X$ and any $x_0\in X$, then $(\overline{x_0\cdot G}, x_0)$ is a $G$-ambit. For ambits $(X,x_0)$ and $(Y, y_0)$, a map of $G$-ambits is a $G$-map $f: X\rightarrow Y$ with $f(x_0) = y_0$. Notice that if there is a map of $G$-ambits $f: (X,x_0)\rightarrow (Y, y_0)$, it must be unique since $f$ is determined on the dense set $x_0\cdot G$. The greatest ambit $(S(G), 1)$ is characterized by being universal for the class of $G$-ambits, i.e.\ for any $G$-ambit $(X,x_0)$, there is a map of $G$-ambits $f: (S(G), 1)\rightarrow (X,x_0)$. Since maps between ambits are unique, the greatest ambit, should it exist, is unique up to a unique isomorphism of $G$-ambits.  

\begin{theorem}
For any topological group $G$, there exists a greatest $G$-ambit $(S(G), 1)$.
\end{theorem}

\begin{proof}
Notice first that any topological space with a dense subset of cardinality $\kappa$ has size at most $2^{2^\kappa}$, so we may assume that there is a set (as opposed to a proper class) containing every isomorphism type of $G$-ambit.

Let $\{(X_i, x_0^i): i\in I\}$ list every isomorphism type of $G$-ambit, and form $Z = \prod_i X_i$. Let $G$ act on $Z$ coordinatewise; this action is continuous. Set $1 = (x_0^i)_{i\in I}$, and let $S(G) = \overline{1\cdot G}$. Then $(S(G), 1)$ is a $G$-ambit. Let $\pi_i: Z\rightarrow X_i$ be projection onto the $i^{th}$ coordinate. Then $\pi_i|_{S(G)}$ is a map of $G$-ambits onto $(X_i, x_0^i)$. It follows that $(S(G), 1)$ is the greatest $G$-ambit.
\end{proof}  

A more constructive proof using functional analysis can be found in [KPT]. As we will give a detailed construction for closed subgroups of $S_\infty$ in section 6, the abstract approach is sufficient here. Our construction is also the reason why we are opting to use right actions as opposed to left actions, which are more commonly found in the literature. In order to study the properties of the greatest ambit, we need to develop some of the theory of topological semigroups.

We say that a semigroup $S$ is a \emph{left-topological} semigroup if $S$ is a compact topological space in which left multiplication is continuous, i.e.\ for each $s\in S$, the map $\lambda_s: S\rightarrow S$ with $\lambda_s(t) = st$ is continuous. A \emph{right ideal} of $S$ is a non-empty subset $I\subseteq S$ with $IS\subseteq I$. A right ideal is \emph{minimal} if it does not properly contain any right ideals. Equivalently, $I$ is a minimal right ideal iff $xS = I$ for every $x\in I$; in particular, since $xS = \lambda_x(S)$ and $S$ is compact, minimal right ideals are always closed (closed always refers to the topology; we will write $I^2\subseteq I$ when we mean closed with respect to the operation).  

\begin{prop}
If $S$ is a left-topological semigroup, then there is a minimal right ideal.
\end{prop}

\begin{proof}
Apply Zorn's lemma to the collection of closed right ideals (which is non-empty as $S$ is a member). Let $I$ be minimal. Then for every $x\in I$, $xS$ is a closed right ideal, and $xS\subseteq I$. Hence $xS = I$.
\end{proof}

If $S$ is a semigroup, $u\in S$ is an $\emph{idempotent}$ if $u^2 = u$.

\begin{prop}
If $I$ is any minimal right ideal of a left-topological semigroup $S$, then $I$ contains an idempotent.
\end{prop}

\begin{proof}
Consider the collection $\mathcal{F} = \{J\subseteq I: J\neq \emptyset\text{ closed and } J^2\subseteq J\}$. This collection is nonempty as $I\in \mathcal{F}$. Applying Zorn, let $M$ be minimal. Pick $w\in M$; then $wM\in \mathcal{F}$ and $wM\subseteq M$, so $wM = M$. Let $X := M\cap \lambda_w^{-1}(w)$. We see that $\emptyset \neq X\in \mathcal{F}$ and $X\subseteq M$, so $X = M$, $w\in X$, and $w^2 = w$.
\end{proof}

Now if $I$ is a minimal right ideal and $u\in I$ is an idempotent, then $u$ is a left identity for $I$. To see this, notice that $I = uI$, so if $x\in I$, we may write $x = uy$ for some $y\in I$. Then $ux = u^2y = uy = x$. 

Now let us consider the greatest ambit $S(G)$; we can give $S(G)$ a left-topological semigroup structure as follows. If $x, y\in S(G)$ and we want to define $xy$, consider the orbit closure $X := \overline{x\cdot G}$. Then $(X, x)$ is an ambit, and there is a unique map of $G$-ambits $f_x: (S(G), 1)\rightarrow (X, x)$. Then we can define $xy = f_x(y)$. Associativity follows once we note that $f_x\circ f_y = f_{xy}$. Notice that $f_x$ is continuous and for $g\in G$, we have $x\cdot g = f_x(1\cdot g)$. It is a fact that the map $g\rightarrow 1\cdot g$ is a homeomorphism of $G$ onto its image, and many authors identify $G$ as a subspace of $S(G)$, as we will do in section 6.

Suppose $M\subseteq S(G)$ is a minimal $G$-flow and pick any $x\in M$. Then we have $xy = f_x(y)\in M$ for any $y\in S(G)$, so $M$ is a minimal right ideal. Conversely, if $M$ is a minimal right ideal, pick any $x\in M$. Then $M = xS(G)$, and as $1\cdot G$ is dense in $S(G)$, $x\cdot G$ is dense in $M$, so $M$ is a minimal $G$-flow. 

\begin{prop}
Let $(S(G), 1)$ be the greatest ambit, and suppose $M\subseteq S(G)$ is a minimal right ideal (i.e.\ minimal $G$-flow). Suppose $\mu: M\rightarrow M$ is a $G$-map, and fix $x,y\in M$. Then:
\begin{enumerate}
\item
$\mu(xy) = \mu(x)y$,
\item
There is $w\in M$ such that for all $x\in M$, $\mu(x) = wx$,
\item
$\mu$ is an isomorphism.
\end{enumerate}
\end{prop}

\begin{proof}
For the first claim, use the fact that left multiplication is continuous and that $1\cdot G$ is dense in $S(G)$. For the second claim, let $u\in M$ be an idempotent, and set $w = \mu(u)$. Pick $x\in M$; then $\mu(x) = \mu(ux) = \mu(u)x = wx$. 

Since $M$ is minimal, $\mu$ is surjective. To see that $\mu$ is an isomorphism, it is enough to show that $\mu$ is injective (since $M$ is compact). Notice that $wM = M$, so pick $v\in M$ with $wv = u$. But $vM = M$ as well, and $wvx = ux = x$. As left multiplication by $wv$ is injective, so is left multiplication by $w$, i.e.\ $\mu$ is injective.
\end{proof}

\begin{theorem}
Let $M$ be a minimal subflow of the greatest ambit $(S(G), 1)$, and let $N$ be any universal minimal flow for $G$. Then $M\cong N$. Hence the universal minimal flow $M(G)$ is unique up to isomorphism.
\end{theorem}

\begin{proof}
Since $M$ and $N$ are universal, let $f: M\rightarrow N$ and $g: N\rightarrow M$ be $G$-maps. By the above proposition, we have that $f\circ g$ is an isomorphism. Hence $f$ is injective; as $N$ is minimal, $f$ is surjective. As $M$ is compact, $f$ is an isomorphism.
\end{proof}

See Auslander [A] for a more detailed exposition of topological dynamics and the universal minimal flow. See [HS] or [EEM] for more on topological semigroups.

\subsection{Filters, Ultrafilters, and the $\beta$-compactification}

Let $X$ be a set. A \emph{filter} on $X$ is a collection $\mathcal{F}\subseteq \mathcal{P}(X)$ satisfying the following:

\begin{itemize}
\item
$\mathcal{F}$ is nontrivial: $X\in \mathcal{F}$ and $\emptyset \not\in \mathcal{F}$,
\item
$\mathcal{F}$ is upwards closed: if $A\in \mathcal{F}$ and $A\subseteq B$, then $B\in \mathcal{F}$,
\item
$\mathcal{F}$ is closed under finite intersections: if $A,B\in \mathcal{F}$, then $A\cap B\in \mathcal{F}$.
\end{itemize}

Notice that the union of a chain of filters is also a filter, so by Zorn's Lemma, every filter is contained in a maximal filter. These are called \emph{ultrafilters}. Equivalently, ultrafilters are those filters which contain $A$ or $X\setminus A$ for every $A\subseteq X$. The prototypical example of an ultrafilter is a \emph{principal} ultrafilter, one of the form $p_x := \{A\subseteq X: x\in A\}$ for some fixed $x\in X$. 

Let $X$ and $Y$ be sets, $f: X\rightarrow Y$ any function, $g: X\rightarrow Y$ a surjective function, $\mathcal{F}$ a filter on $X$, and $\mathcal{G}$ a filter on $Y$. Then $f(\mathcal{F})$, the \emph{push forward} of $\mathcal{F}$, is the filter on $Y$ with $A\in f(\mathcal{F})$ iff $f^{-1}(A)\in \mathcal{F}$. The \emph{pre-image filter} $g^{-1}(\mathcal{G})$ is the filter on $X$ generated by the sets $g^{-1}(B)$ for $B\in \mathcal{G}$. The push forwards of ultrafilters are ultrafilters, but pre-images of ultrafilters are typically just filters.

The dual notion to a filter is an \emph{ideal}, a collection $\mathcal{I}\subseteq \mathcal{P}(X)$ which is nontrivial ($\emptyset\in\mathcal{I}$ and $X\not\in \mathcal{I}$), downwards closed, and closed under finite unions. $\mathcal{I}$ is an ideal iff $\{A\subseteq X: X\setminus A\in \mathcal{I}\}$ is a filter; we call this the \emph{dual filter} of $\mathcal{I}$. Every ideal is contained in a maximal ideal, and $\mathcal{I}$ is a maximal ideal iff $\{A\subset X: X\setminus A\in \mathcal{I}\}$ is an ultrafilter. 

Denote the space of ultrafilters on $X$ by $\beta X$; we endow $\beta X$ with the topology whose basic open sets are of the form $\overline{A} := \{p\in \beta{X}: A\in p\}$ for $A\subseteq X$. Notice that each of these basic open sets is closed, $\overline{A} = cl(A)$, and $\overline{A} \sqcup \overline{X\setminus A} = \beta X$. We can identify $X$ as a subspace of $\beta X$ by identifying each $x\in X$ with the principal ultrafilter $p_x$. Notice that $\{p_x\} = \overline{\{x\}}$, so under the identification, $X$ is an open, discrete subspace of $\beta X$.  

Let us show that $\beta X$ is the \emph{Stone-\v{C}ech compactification} of $X$, that is to say $\beta X$ has the properties guaranteed by Proposition 2.6 and Theorem 2.7.
\begin{prop}
$\beta X$ is compact Hausdorff.
\end{prop}

\begin{proof}
$\beta X$ is easily seen to be Hausdorff. As for compactness, suppose not. Use Zorn's Lemma to obtain a maximal open cover $\mathcal{C}$ with respect to the properties that each $U\in \mathcal{C}$ is of the form $\overline{A}$ for some $A\subseteq X$ and such that $\mathcal{C}$ contains no finite subcover. Let $\mathcal{I} = \{A\subseteq X: \overline{A}\in \mathcal{C}\}$. I claim $\mathcal{I}$ is an ideal. Certainly $X\not\in \mathcal{I}$, and $\emptyset\in \mathcal{I}$ by maximality of $\mathcal{C}$; maximality also gives downward closure. If $A, B\in \mathcal{I}$, then $A\cup B\in \mathcal{I}$, for if there were a finite cover from $\mathcal{C}\cup \{\overline{A\cup B}\}$, then a finite cover could be found using $\overline{A}$ and $\overline{B}$.

Let $p\in \beta X$ be any ultrafilter extending the dual filter of $\mathcal{I}$. Then every member of $p$ is not in $\mathcal{I}$. It follows that $p$ is not covered by $\mathcal{C}$, a contradiction.
\end{proof}

$\beta X$ also satisfies the following fundamental universal property:

\begin{theorem}
For any map $f: X\rightarrow Y$ with $Y$ compact Hausdorff, there is a unique continuous extension $\tilde{f}: \beta X\rightarrow Y$, i.e.\ the following diagram commutes:
\begin{center}
\begin{tikzpicture}[node distance=3cm, auto]
\node (A) {$X$};
\node (B) [right of=A] {$Y$};
\node (C) [node distance=1.8 cm, above of=A] {$\beta X$};
\draw[->] (A) to node [swap] {$f$} (B);
\draw[->] (A) to node {$i$} (C);
\draw[->] (C) to node {$\tilde{f}$} (B);
\end{tikzpicture}
\end{center}
\end{theorem}

\begin{proof}
For $p\in \beta X$, we define $\tilde{f}(p)$ to be the unique point in $Y$ such that for every neighborhood $U\ni \tilde{f}(p)$ in $Y$, then $f^{-1}(U)$ is in $p$. We need to show that such $\tilde{f}(p)$ actually does exist and is unique. Uniqueness follows as $Y$ is Hausdorff. For existence, suppose each $y\in Y$ had a neighborhood $U_y\ni y$ with $f^{-1}(U_y)$ not in $p$. Then $\{U_y: y\in Y\}$ is an open cover of $Y$, so let $\{U_{y_1},...,U_{y_k}\}$ be a finite subcover. Then $(X\setminus f^{-1}(U_{y_1}))\cap\cdots \cap (X\setminus f^{-1}(U_{y_k})) = \emptyset\in p$, a contradiction.
\end{proof}

Let $G$ be an infinite discrete group, and form $\beta G$. We can give $\beta G$ a left-topological semigroup structure as follows:

\begin{itemize}
\item
For each fixed $g\in G$, the map $h\rightarrow hg$ has a unique continuous extension to $\beta G$,
\item
For each fixed $p\in \beta G$, the map $h\rightarrow ph$ has a unique continuous extension to $\beta G$.
\end{itemize}

Associativity must be verified, but is straightforward by repeated application of Theorem 2.7. Note that it was arbitrary whether we started with left or right multiplication; however, you can only choose one of left or right multiplication to be continuous. Here, we have chosen a semigroup structure where right multiplication by elements of $G$ is continuous and left multiplication by any $p\in \beta G$ is continuous. We can also identify elements of $\beta G$ with ultrafilters on $G$. For $A\subseteq G$ and $p,q\in \beta G$, we have $A\in pq$ iff $\{h\in G: Ah^{-1}\in p\}\in q$. In particular, if $p\in \beta G$ and $g\in G$, we have $A\in pg$ iff $Ag^{-1}\in p$. For more on semigroup compactifications, see [HS] or [EEN].

We conclude this subsection by discussing some of the topological properties of the space $\beta X$ for $X$ a discrete space.

\begin{prop}
If $U\subseteq \beta X$ is open, then $\overline{U}$ is open. 
\end{prop}

Spaces with this property are called \emph{extremely disconnected}.

\begin{proof}
Set $A = U\cap X$. If $p\in U$, then as $U$ is open, $p\in \overline{B}$ for some $B\subseteq A$. So $\overline{A}$ is a closed subspace containing $U$, and any such subspace must contain $\overline{A}$. Hence $\overline{A} = \overline{U}$, and $\overline{A}$ is open. 
\end{proof}

\begin{theorem}
Let $Y$ be a regular extremely disconnected topological space. Then $Y$ embeds no infinite compact metric spaces. In particular, $\beta X$ embeds no infinite compact metric spaces.
\end{theorem}

\begin{proof}
As any infinite compact metric space contains an infinite convergent sequence, we will show that $Y$ has no infinite convergent sequences. Suppose to the contrary that we had $(x_i)_{i\in \mathbb{N}}\subseteq Y$ with $i\neq j \Rightarrow x_i\neq x_j$ and $\lim_{n\to\infty} x_n = x$. We may assume $x_n\neq x$ for all $n$. Set $A = \{x_i: i \text{ even}\}$, $B = \{x_i: i \text{ odd}\}$. Notice that $\overline{A} = A\cup \{x\}$ and $\overline{B} = B\cup \{x\}$.

Using regularity, we will inductively define open sets $\{U_n: n\in \mathbb{N}\}$ as follows. Let $U_1$ contain $x_1$ with $\overline{U_1}\cap \overline{A} = \emptyset$. Now define $U_2$ containing $x_2$ with $\overline{U_2}\cap (\overline{B}\cup\overline{U_1}) = \emptyset$. If $n$ is odd and $U_1,...,U_{n-1}$ are defined, pick $U_n$ containing $x_n$ with $\overline{U_n}\cap (\overline{A}\cup\overline{U_2}\cup\overline{U_4}\cup\cdots\cup\overline{U_{n-1}}) = \emptyset$. If $n$ is even, pick $U_n$ containing $x_n$ with $\overline{U_n}\cap (\overline{B}\cup\overline{U_1}\cup\overline{U_3}\cup\cdots\cup \overline{U_{n-1}}) = \emptyset$.

Let $U = \bigcup_{n} U_{2n}$, $V = \bigcup_{n} U_{2n-1}$. Notice that $U\cap \overline{V} = \overline{U}\cap V = \emptyset$. As $Y$ is extremely disconnected, $\overline{U}$ and $\overline{V}$ are open. If $\overline{U}\cap\overline{V}$ were nonempty, then $\overline{U}\cap\overline{V}$ is nonempty open and must meet $U$, a contradiction. So $\overline{U}\cap\overline{V} = \emptyset$, but $x\in \overline{U}\cap\overline{V}$. Hence the sequence $(x_i)_{i\in\mathbb{N}}$ cannot converge to $x$.
\end{proof}

As we will often study spaces by mapping them into $\beta X$ for some discrete space $X$, we also need the following folklore fact (see [W], p.\ 166):

\begin{prop}
The continuous image of a compact metrizable space in a Hausdorff space is metrizable.
\end{prop}

\begin{proof}

Let $X$ be compact metrizable with compatible metric $d_X$, $Y$ Hausdorff, and $f: X\rightarrow Y$ continuous. We may assume $f$ is surjective. On $Y$, define $d_Y: Y\times Y\rightarrow \mathbb{R}$ by

$$d_Y(a,b) = \max_{x\in f^{-1}(a)}\left(d_X(x,f^{-1}(b))\right) + \max_{y\in f^{-1}(b)}\left(d_X(f^{-1}(a),y)\right)$$

\vspace{3 mm}
This is a compatible metric for $Y$.
\end{proof}
\begin{rem}
We can view the above proof as identifying each $y\in Y$ with $f^{-1}(y)\in K(X)$, where $K(X)$ is the Vietoris space of compact subsets of $X$. With this identification, a compatible metric on $Y$ is just the Hausdorff metric on $K(X)$.
\end{rem}

\section{\fr structures}

We now move towards the case we will consider, where $G$ is a closed subgroup of $S_\infty$. In this section, we describe a canonical way of viewing any such group. Recall that $S_\infty$ is the group of permutations of $\mathbb{N}$ endowed with the pointwise convergence topology; a basis of open sets at the identity is given by $G_n$, the pointwise stabilizer of $\{1, 2,...,n\}$. A compatible left-invariant metric is given by $d(g,h) = 1/n$ iff $n$ is least with $g^{-1}h\not\in G_n$.

A \emph{language} $L = \{R_i: i\in I\}\cup \{f_j: j\in J\}\cup \{c_k: k\in K\}$ is a collection of relation, function and constant symbols. Each relation symbol $R_i$ has an \emph{arity} $n_i\in \mathbb{N}$, as does each function symbol $f_j$. An \emph{$L$-structure} $\mathbf{A} = \langle A, R_i^\mathbf{A}, f_j^\mathbf{A}, c_k^\mathbf{A}\rangle$ consists of a set $A$, relations $R_i^\mathbf{A}\subseteq A^{n_i}$, functions $f_j^{\mathbf{A}}: A^{n_j}\rightarrow A$, and constants $c_k^\mathbf{A}\in A$; we say that $\mathbf{A}$ is an $L$-structure on $A$. If $\mathbf{A},\mathbf{B}$ are $L$-structures, then $g: \mathbf{A}\rightarrow \mathbf{B}$ is an \emph{embedding} if $g$ is a map from $A$ to $B$ such that $R_i^\mathbf{A}(x_1,...,x_{n_i})\Leftrightarrow R_i^\mathbf{B}(g(x_1),...,g(x_{n_i}))$, $g(f_j^\mathbf{A}(x_1,...,x_{n_j})) = f_j^\mathbf{B}(g(x_1),...,g(x_{n_j}))$, and $g(c_k^\mathbf{A}) = c_k^\mathbf{B}$ for all relations, functions, and constants, respectively. If there is an embedding $g: \mathbf{A} \rightarrow \mathbf{B}$, we say that $\mathbf{B}$ \emph{embeds} $\mathbf{A}$. An \emph{isomorphism} is a bijective embedding, and an \emph{automorphism} is an isomorphism between a structure and itself. If $A\subseteq B$, then we say that $\mathbf{A}$ is a \emph{substructure} of $\mathbf{B}$, written $\mathbf{A}\subseteq \mathbf{B}$, if the inclusion map is an embedding. $\mathbf{A}$ is finite, countable, etc.\ if $A$ is.

Let $\mathbf{K}$ be a countably infinite $L$-structure. We say that $\mathbf{K}$ is \emph{locally finite} if there are finite substructures $\mathbf{A}_n\subseteq \mathbf{K}$ with $\mathbf{A}_n\subseteq \mathbf{A}_{n+1}$ and $\mathbf{K} = \bigcup_n \mathbf{A}_n$. Then $\bigcup_n \mathbf{A}_n$ is said to be an \emph{exhaustion} of $\mathbf{K}$. We set $\fin{\mathbf{K}}$ to be the set of finite substructures of $\mathbf{K}$, and we set $\mathcal{K} = \age{\mathbf{K}}$, the \emph{age} of $\mathbf{K}$, to be the class of finite $L$-structures which embed into $\mathbf{K}$, i.e.\ those structures isomorphic to some structure in $\fin{\mathbf{K}}$. It is natural to ask which classes of finite structures are the age of a countably infinite locally finite structure. If $\mathcal{K}$ is a class of finite structures, we call $\mathcal{K}$ an \emph{age class} if $\mathcal{K}$ satisfies the following:

\begin{itemize}
\item
$\mathcal{K}$ is closed under isomorphism, contains countably many isomorphism types, and contains structures of arbitrarily large finite cardinality
\item
$\mathcal{K}$ has the \emph{Hereditary Property} (HP): if $\mathbf{B}\in\mathcal{K}$ and $\mathbf{A}\subseteq \mathbf{B}$, then $\mathbf{A}\in\mathcal{K}$,
\item
$\mathcal{K}$ satisfies the \emph{Joint Embedding Property} (JEP): if $\mathbf{A},\mathbf{B}\in \mathcal{K}$, then there is $\mathbf{C}\in\mathcal{K}$ which embeds both $\mathbf{A}$ and $\mathbf{B}$.
\end{itemize}

\begin{prop}
$\mathcal{K}$ is an age class iff $\mathcal{K} = \age{\mathbf{K}}$ for $\mathbf{K}$ some countably infinite locally finite structure. 
\end{prop}

\begin{proof}
Certainly if $\mathbf{K}$ is a countably infinite locally finite structure, then $\age{\mathbf{K}}$ is an age class. Conversely, suppose $\mathcal{K}$ is an age class. Let $(\mathbf{B}_n: n\in \mathbf{N})$ be an enumeration of the isomorphism types in $\mathcal{K}$. We will define for each $n\in \mathbb{N}$ a structure $\mathbf{A}_n\in\mathcal{K}$ as follows. Set $\mathbf{A}_1 = \mathbf{B}_1$. If $\mathbf{A}_n$ has been defined, pick $\mathbf{A}_{n+1}\supseteq \mathbf{A}_n$ which witnesses JEP for $\mathbf{A}_n$ and $\mathbf{B}_{n+1}$. Now set $\mathbf{K} = \bigcup_n \mathbf{A}_n$; we see that $\age{\mathbf{K}} = \mathcal{K}$.
\end{proof} 

\begin{rem}
It should be noted that the $\mathbf{K}$ contructed in the proof of Proposition 3.1 is not unique.
\end{rem}

A countably infinite locally finite structure $\mathbf{K}$ is a \emph{\fr}structure if $\mathbf{K}$ is \emph{ultrahomogeneous}:

\begin{itemize}
\item
For any $\mathbf{A}\in \fin{\mathbf{K}}$ and any embedding $g: \mathbf{A}\rightarrow \mathbf{K}$, there is an automorphism of $\mathbf{K}$ extending $g$. 
\end{itemize}

Another useful, equivalent definition (Proposition 3.2) is that $\mathbf{K}$ is a \fr structure iff it is locally finite and for every $\mathbf{A}\subseteq \mathbf{B}\in \age{\mathbf{K}}$, every embedding $g: \mathbf{A}\rightarrow \mathbf{K}$ can be extended to an embedding $h: \mathbf{B}\rightarrow \mathbf{K}$. This is often called the \emph{extension property} for $\mathbf{K}$. Often embeddings $g: \mathbf{A}\rightarrow \mathbf{K}$ with $\mathbf{A}\in \fin{\mathbf{K}}$ are called \emph{partial isomorphisms}. In this case, we let $g^{-1}: g(\mathbf{A})\rightarrow \mathbf{K}$ be defined via $g^{-1}\circ g(a) = a$ for each $a\in A$.

\begin{prop} 
For $\mathbf{K}$ a countably infinite locally finite $L$-structure, $\mathbf{K}$ is a \fr structure iff $\mathbf{K}$ has the extension property.
\end{prop}

\begin{proof}
$(\Rightarrow)$ Fix $\mathbf{A}\subseteq \mathbf{B}\in \mathcal{K}$ and an embedding $g: \mathbf{A}\rightarrow \mathbf{K}$, we may assume that $\mathbf{A}\subseteq \mathbf{K}$ and that $g$ is the inclusion map. Let $f: \mathbf{B}\rightarrow \mathbf{K}$ be any embedding. Then $f|_\mathbf{A}$  can be extended to an isomorphism $h$ of $\mathbf{K}$. Then $h^{-1}\circ f$ extends $g$.

$(\Leftarrow)$ We use a common method known as the back-and-forth method. Suppose $\mathbf{A}\in \fin{\mathbf{K}}$ and $g: \mathbf{A}\rightarrow \mathbf{K}$ is a partial isomorphism. Fix an exhaustion $\bigcup_n \mathbf{A}_n = \mathbf{K}$. Set $g = g_1$. Find $k_2$ large enough so that $g_1(\mathbf{A}_1)\subseteq \mathbf{A}_{k_2}$. Use the extension property to extend $g_1^{-1}$ to an embedding $g_2: \mathbf{A}_{k_2}\rightarrow \mathbf{K}$. Find $k_3\in\mathbb{N}$ large enough with $g_2(\mathbf{A}_{k_2})\subseteq \mathbf{A}_{k_3}$. Use the extension property once again to extend $g_2^{-1}$ to $g_3: \mathbf{A}_{k_3}\rightarrow \mathbf{K}$. We continue going back and forth in this fashion. Set $h = \bigcup_n g_{2n-1}$. Then $h:\mathbf{K}\rightarrow \mathbf{K}$ is an isomorphism with $h^{-1} = \bigcup_n g_{2n}$. 
\end{proof}

The back and forth method is also used to show that if two \fr structures have the same age, then they are isomorphic. However, not all age classes are the ages of \fr structures. A class of finite structures $\mathcal{K}$ is a $\emph{\fr class}$ if $\mathcal{K}$ is an age class which additionally satisfies the \emph{Amalgamation Property} (AP):

\begin{itemize}
\item
If $\mathbf{A}, \mathbf{B}, \mathbf{C}\in \mathcal{K}$ and $f: \mathbf{A}\rightarrow \mathbf{B}$ and $g: \mathbf{A}\rightarrow \mathbf{C}$ are embeddings, there is $\mathbf{D}\in \mathcal{K}$ and embeddings $r: \mathbf{B}\rightarrow \mathbf{D}$ and $s:\mathbf{C}\rightarrow\mathbf{D}$ with $r\circ f = s\circ g$.
\end{itemize}

\begin{rem}
It is enough in the definition of AP to take $f$, $g$, and $r$ to be inclusion maps.
\end{rem}

We will need the notion of isomorphic inclusion pairs for the next proof. If $\mathbf{A},\mathbf{B},\mathbf{C},\mathbf{D}$ are $L$-structures with $\mathbf{A}\subseteq \mathbf{B}$ and $\mathbf{C}\subseteq \mathbf{D}$, we say that $(\mathbf{A}\subseteq \mathbf{B})\cong (\mathbf{C}\subseteq \mathbf{D})$ if there is an isomorphism $f: \mathbf{B}\rightarrow \mathbf{D}$ with $f|_{\mathbf{A}}$ an isomorphism of $\mathbf{A}$ and $\mathbf{C}$.

\begin{prop}
$\mathcal{K}$ is a \fr class iff $\mathcal{K} = \age{\mathbf{K}}$ for some \fr structure $\mathbf{K}$.
\end{prop}

\begin{proof}
If $\mathbf{K}$ is a \fr structure with $\age{\mathbf{K}} = \mathcal{K}$, let $\mathbf{A}\in\fin{\mathbf{K}}$. Suppose $\mathbf{B},\mathbf{C}\in \mathcal{K}$ with $\mathbf{A}\subseteq \mathbf{B}$ and $\mathbf{A}\subseteq \mathbf{C}$. Use the extension property to find embeddings $f: \mathbf{B}\rightarrow \mathbf{K}$ and $g: \mathbf{C}\rightarrow \mathbf{K}$ extending the inclusion map of $\mathbf{A}$. Now as $\mathbf{K}$ is locally finite, find $\mathbf{D}\in \fin{\mathbf{K}}$ with $f(\mathbf{B})\cup g(\mathbf{C})\subseteq \mathbf{D}$. This is enough to verify that $\mathcal{K}$ has the AP.

Conversely, suppose $\mathcal{K}$ is a \fr class. Let $\langle (\mathbf{B}_n\subseteq \mathbf{C}_n): n\in \mathbb{N}\rangle$ list all isomorphism types of inclusion pairs such that each isomorphism type appears infinitely often. For $n\in \mathbb{N}$, we will choose $\mathbf{A}_n\in\mathcal{K}$ as follows. Pick $\mathbf{A}_1\in \mathcal{K}$ arbitrarily. If $\mathbf{A}_n$ has been defined, let $f_1,...,f_k$ list the embeddings of $\mathbf{B}_n$ into $\mathbf{A}_n$. We will build a sequence of structures $\mathbf{A}_n^i\in \mathcal{K}$ with $\mathbf{A}_n:= \mathbf{A}_n^0\subseteq \mathbf{A}_n^1\subseteq \mathbf{A}_n^2\subseteq\cdots\subseteq \mathbf{A}_n^k := \mathbf{A}_{n+1}$. If $\mathbf{A}_n^i$ has been defined, then pick $\mathbf{A}_n^{i+1}\supseteq \mathbf{A}_n^i$ which witnesses the AP for $\mathbf{B}_n\subseteq \mathbf{C}_n$ and $f_{i+1}: \mathbf{B}_n\rightarrow \mathbf{A}_n^i$. If no embeddings $f_i$ exist, then pick $\mathbf{A}_{n+1}\supseteq \mathbf{A}_n$ witnessing JEP for $\mathbf{A}_n$ and $\mathbf{B}_n$. Set $\mathbf{K} = \bigcup_{n} \mathbf{A}_n$. Then $\age{\mathbf{K}} = \mathcal{K}$, and $\mathbf{K}$ satisfies the extension property and is $\sigma$-finite. Hence $\mathbf{K}$ is a \fr structure.
\end{proof}

The $\mathbf{K}$ constructed in Proposition 3.3 is unique up to isomorphism as can be shown using the back and forth method; we write $\mathbf{K} = \flim{\mathcal{K}}$, the \emph{\fr limit} of $\mathcal{K}$. 

We can also define \fr limits of more general classes. If $\mathbf{K}$ is a countable structure and $\mathcal{K}\subseteq \age{\mathbf{K}}$ is closed under isomorphism, we say that $\mathbf{K}$ is $\mathcal{K}$-homogeneous if any partial isomorphism of structures in $\mathcal{K}$ can be extended to an automorphism of $\mathbf{K}$. Most often, we will use this added generality when $\mathcal{K}$ is a \emph{Fra\"iss\'e--HP class}; i.e.\ a class of finite structures which satisfies every condition of being a \fr class except possibly the Hereditary Property. If $\mathcal{K}$ is a class of structures which is not necessarily hereditary, let $\mathcal{K}\!\!\downarrow := \{\mathbf{A}: \exists \mathbf{B}\in \mathcal{K} (\mathbf{A}\subseteq \mathbf{B})\}$. Now if $\mathcal{K}$ is a Fra\"iss\'e--HP class, we can use the exact same proof as in Proposition 3.3 to show that up to isomorphism, there is a unique countably infinite locally finite structure with age $\mathcal{K}\!\!\downarrow$ which is $\mathcal{K}$-homogeneous; we will also call this the \fr limit. The proof shows that if $\mathbf{K} = \flim{\mathcal{K}}$, then $\mathbf{K}$ has an exhaustion $\mathbf{K} = \bigcup_n \mathbf{A}_n$ with $\mathbf{A}_n\in \mathcal{K}$. We will call this a $\mathcal{K}$-exhaustion.

Our interest in \fr structures stems from the following:

\begin{theorem}
$G$ is a closed subgroup of $S_\infty$ iff $G$ is the automorphism group of a relational \fr structure on $\mathbb{N}$.
\end{theorem}

\begin{proof}
If $\mathbf{K}$ is a relational \fr structure and $G = \aut{\mathbf{K}}$, then if $g_n\in G$ and $g_n\rightarrow g$ with $g\in S_\infty$, then $g$ must also be an automorphism of $\mathbf{K}$ and hence in $G$. Conversely, suppose $G$ is a closed subgroup of $S_\infty$. For every $\overline{a}\in {}^{<\omega}\mathbb{N}$, introduce a relational symbol $R_{\overline{a}}$ of arity $\mathrm{len}(\overline{a})$, and let $L = \{R_{\overline{a}}: \overline{a}\in {}^{<\omega}\mathbb{N}\}$. Give $\mathbb{N}$ an $L$-structure by declaring that $R_{\overline{a}}(b_1,...,b_n)$ iff there is $g\in G$ with $g(a_i) = b_i$ for each $i\leq n$. Then $\mathbf{K} = \langle \mathbb{N}, \{R_{\overline{a}}^\mathbf{K}: \overline{a}\in {}^{<\omega}\mathbb{N}\}\rangle$ is a \fr structure with $\aut{\mathbf{K}} = G$.
\end{proof}

For a more detailed exposition of \fr theory, see [Ho].

\section{Structural Ramsey Theory}
In this section, we introduce some of the ideas underlying structural Ramsey theory. However, we begin with a discussion of Ramsey theory for embeddings, as this is what we will use in the rest of the paper. Proposition 4.4 makes the connection between the structural and embedding versions explicit.

A \emph{partial $k$-coloring} $\gamma$ of a set $X$ is a function $\gamma: Y\rightarrow [k]$, where $Y\subseteq X$ and $[k] = \{1,2,...,k\}$. A coloring $\gamma$ of $X$ is \emph{full} if $\mathrm{dom}(\gamma) = X$. We will often write $\gamma_i$ for $\gamma^{-1}(i)$. If $\dom{\gamma}$ is unspecified, then $\gamma$ is presumed to be a full coloring. If $\gamma$ is a coloring of $X$ and $Y\subseteq \mathrm{dom}(\gamma)$, we say that $Y$ is \emph{monochromatic} if $Y\subseteq \gamma_i$ for some $i$.

If $\mathbf{A},\mathbf{B}$ are $L$-structures, write $\emb{\mathbf{A},\mathbf{B}}$ for the set of embeddings from $\mathbf{A}$ to $\mathbf{B}$, and write $\mathbf{A}\leq \mathbf{B}$ if $\emb{\mathbf{A},\mathbf{B}}\neq\emptyset$. If $\mathcal{C}$ is a class of finite $L$-structures, we say that $\mathbf{A}\in \mathcal{C}$ is a \emph{Ramsey object} if for any $\mathbf{B}\in \mathcal{C}$ with $\mathbf{A}\leq \mathbf{B}$, there is $\mathbf{C}\in \mathcal{C}$, $\mathbf{A}\leq \mathbf{C}$, such that for any full 2-coloring of $\emb{\mathbf{A},\mathbf{C}}$, there is $f\in \emb{\mathbf{B},\mathbf{C}}$ with $f\circ \emb{\mathbf{A},\mathbf{B}}$ monochromatic. We say that $\mathcal{C}$ has the \emph{Ramsey Property} (RP) if each $\mathbf{A}\in \mathcal{C}$ is a Ramsey object. The choice of $2$ colors is arbitrary; a straightforward induction on the number of colors shows that if $\mathbf{A}\in \mathcal{C}$ is a Ramsey object, then for any $k\geq 2$ and any $\mathbf{B}\in\mathcal{C}$ with $\mathbf{A}\leq \mathbf{B}$, there is a $\mathbf{C}\in \mathcal{C}$ such that for any $k$-coloring of $\emb{\mathbf{A},\mathbf{C}}$, there is $f\in \emb{\mathbf{B},\mathbf{C}}$ with $f\circ \emb{\mathbf{A},\mathbf{B}}$ monochromatic.

Once again, we are using an embedding version of Ramsey object/Ramsey property, as opposed to the structural version defined in the introduction. A useful translation between the two versions is as follows:  suppose $\gamma: \emb{\mathbf{A},\mathbf{C}}\rightarrow [r]$ is a coloring which additionally has $\gamma(f) = \gamma(g)$ whenever $f = g\circ h$ for some $h\in \aut{\mathbf{A}}$. Let us call such a $\gamma$ a \emph{structural coloring}. Then we may define $\gamma': \binom{\mathbf{C}}{\mathbf{A}}\rightarrow [r]$ via $\gamma'(\mathbf{A}_0) = \gamma(f)$ for any $f\in \emb{\mathbf{A},\mathbf{C}}$ with $\im{f} = \mathbf{A}_0$. Conversely, suppose $\gamma: \binom{\mathbf{C}}{\mathbf{A}}\rightarrow [r]$ is a coloring. Then we can define $\gamma': \emb{\mathbf{A},\mathbf{C}}\rightarrow [r]$ via $\gamma'(f) = \gamma(\im{f})$. Notice that this $\gamma'$ is a structural coloring. In what follows, should ``embedding'' or ``structural'' not be specified, ``Ramsey'' will always refer to embedding Ramsey. We will borrow the hook-arrow notation used in [MP],
$$\mathbf{C}\hookrightarrow (\mathbf{B})^\mathbf{A}_k$$
to mean that for any full $k$-coloring of $\emb{\mathbf{A},\mathbf{C}}$, there is $f\in \emb{\mathbf{B},\mathbf{C}}$ with $f\circ \emb{\mathbf{A},\mathbf{B}}$ monochromatic. We use the standard arrow notation,
$$\mathbf{C}\rightarrow (\mathbf{B})^\mathbf{A}_k$$
to mean that for any full $k$-coloring of $\binom{\mathbf{C}}{\mathbf{A}}$, there is $\mathbf{B}_0\in \binom{\mathbf{C}}{\mathbf{B}}$ with $\binom{\mathbf{B}_0}{\mathbf{A}}$ monochromatic.

 If $\mathcal{C}$ and $\mathcal{D}$ are classes of structures, we say that $\mathcal{C}$ is \emph{cofinal} in $\mathcal{D}$ if for any $\mathbf{A}\in \mathcal{D}$, there is $\mathbf{B}\in \mathcal{C}$ with $\mathbf{A}\leq \mathbf{B}$. Suppose that $\mathcal{D}$ is the age of a countably infinite locally finite structure $\mathbf{D}$ and that $\mathcal{C}$ is cofinal in $\mathcal{D}$. For $\mathbf{A}\in \mathcal{C}$, we say that $S\subseteq \emb{\mathbf{A}, \mathbf{D}}$ is \emph{thick} if for any $\mathbf{B}\in \mathcal{C}$ with $\mathbf{A}\leq \mathbf{B}$, there is an $f\in \emb{\mathbf{B},\mathbf{D}}$ with $f\circ \emb{\mathbf{A},\mathbf{B}}\subseteq S$. We say a partial coloring $\gamma$ of $\emb{\mathbf{A},\mathbf{D}}$ is \emph{large} if $\mathrm{dom}(\gamma)$ is thick.

\begin{prop}
Suppose $\mathbf{D}$ is a countably infinite locally finite structure, $\mathcal{D} = \age{\mathbf{D}}$, and $\mathcal{C}$ is cofinal in $\mathbf{D}$. Let $\mathbf{A}\in \mathcal{C}$ and fix any $k\geq 2$. Then the following are equivalent:
\begin{enumerate}
\item
$\mathbf{A}$ is a Ramsey object in $\mathcal{C}$,
\item
$\mathbf{A}$ is a Ramsey object in $\mathcal{D}$,
\item
For any full $k$-coloring $\gamma$ of $\emb{\mathbf{A},\mathbf{D}}$, there is some $\gamma_i$ which is thick,
\item
For any large $k$-coloring $\gamma$ of $\emb{\mathbf{A},\mathbf{D}}$, there is some $\gamma_i$ which is thick.
\end{enumerate}
\end{prop}

\begin{proof}
$(1\Leftrightarrow 2)$ and $(4\Rightarrow 3)$ are straightforward. 

For $(2\Rightarrow 4)$, fix $\gamma$ a large $k$-coloring of $\emb{\mathbf{A},\mathbf{D}}$. Say $\mathbf{A}\leq \mathbf{B}\in\mathcal{D}$, and fix $\mathbf{C}\in \mathcal{D}$ for which $\mathbf{C}\hookrightarrow (\mathbf{B})^\mathbf{A}_k$ holds. Since $\gamma$ is large, find $f\in \emb{\mathbf{C},\mathbf{D}}$ with $f\circ \emb{\mathbf{A},\mathbf{C}}\subseteq \mathrm{dom}(\gamma)$. Then find $x\in\emb{\mathbf{B},\mathbf{C}}$ with $f\circ x\circ \emb{\mathbf{A},\mathbf{B}}$ monochromatic. For each $i\leq k$, let $\mathcal{D}_i\subseteq \mathcal{D}$, where $\mathbf{B}\in \mathcal{D}_i$ iff there is $f\in \emb{\mathbf{B},\mathbf{D}}$ with $f\circ \emb{\mathbf{A},\mathbf{B}}\subseteq \gamma_i$. We have just shown that each $\mathbf{B}\geq \mathbf{A}$ is in some $\mathcal{D}_i$. 

Suppose for sake of contradiction that no $\mathcal{D}_i$ was cofinal in $\mathcal{D}$. For each $i\leq k$, pick $\mathbf{A}_i\in\mathbf{D}$, $\mathbf{A}\leq \mathbf{A}_i$, so that any $\mathbf{B}'\in\mathcal{D}$ which embeds $\mathbf{A}_i$ is not in $\mathcal{D}_i$. Now use JEP for $\mathcal{D}$ to find $\mathbf{A}'$ embedding each $\mathbf{A}_i$. But $\mathbf{A}'\in \mathcal{D}_i$ for some $i\leq k$, so this is a contradiction. Now observe that each $\mathcal{D}_i$ is hereditarily closed, so if $\mathcal{D}_i$ is cofinal, then $\mathcal{D}_i = \mathcal{D}$. This means that $\gamma_i$ must be thick, so we are done.

For $(3\Rightarrow 2)$, let $\mathbf{D} = \bigcup_n \mathbf{B}_n$ be an exhaustion with $\mathbf{A}\leq \mathbf{B}_1$. Suppose $\mathbf{B}\in \mathcal{D}$ witnesses the fact that $\mathbf{A}$ is not a Ramsey object. Call a coloring $\gamma$ of $\emb{\mathbf{A},\mathbf{B}_n}$ \emph{bad} if there is no $f\in \emb{\mathbf{B},\mathbf{D}}$ with $f\circ\emb{\mathbf{A},\mathbf{B}}$ monochromatic. So for each $n$, there is a bad $k$-coloring of $\emb{\mathbf{A},\mathbf{B}_n}$. In particular, if $\gamma$ is a bad $k$-coloring of $\emb{\mathbf{A},\mathbf{B}_n}$ and $m\leq n$, the restriction of $\gamma$ to $\emb{\mathbf{A},\mathbf{B}_m}$ is also bad. We can now use K\"onig's lemma to find a bad full $k$-coloring of $\emb{\mathbf{A},\mathbf{D}}$.
\end{proof}

Often, we will use Proposition 4.1 with a \fr structure $\mathbf{K}$, where we can say more.

\begin{lemma}
Let $\mathbf{K}$ be a \fr structure with $\mathcal{K} = \age{\mathbf{K}}$. Suppose $\mathbf{A}, \mathbf{B}\in\mathcal{K}$, and let $f:\mathbf{A}\rightarrow \mathbf{B}$ be an embedding. If $S\subseteq \emb{\mathbf{B}, \mathbf{K}}$ is thick, then $T := \{x\circ f: x\in S\}\subseteq \emb{\mathbf{A},\mathbf{K}}$ is also thick.
\end{lemma}

\begin{proof}
Fix $\mathbf{C}\in\fin{\mathbf{K}}$. By repeated use of the extension property, find $\mathbf{D}\in\fin{\mathbf{K}}$, $\mathbf{C}\subseteq \mathbf{D}$, such that for each $g\in\emb{\mathbf{A},\mathbf{C}}$, there is $h\in \emb{\mathbf{B},\mathbf{D}}$ with $g = h\circ f$ (here we view $\emb{\mathbf{A},\mathbf{C}}\subseteq \emb{\mathbf{A},\mathbf{D}}$ in the natural way). Now as $S$ is thick, find $x\in\emb{\mathbf{D},\mathbf{K}}$ with $x\circ\emb{\mathbf{B},\mathbf{D}}\subseteq S$. Then $x|_{\mathbf{C}}\in \emb{\mathbf{C},\mathbf{K}}$, and $x|_{\mathbf{C}}\circ \emb{\mathbf{A},\mathbf{C}}\subseteq T$.
\end{proof}

\begin{prop}
Let $\mathbf{K}$ be a \fr structure with $\mathcal{K} = \age{\mathbf{K}}$, and suppose $\mathbf{B}\in\mathcal{K}$ is a Ramsey object. Then if $\mathbf{A}\leq \mathbf{B}$, then $\mathbf{A}$ is a Ramsey object.
\end{prop}

\begin{proof}
Let $f: \mathbf{A}\rightarrow \mathbf{B}$ be an embedding, and fix $\gamma$ a full $2$-coloring of $\emb{\mathbf{A}, \mathbf{K}}$. Let $\delta$ be the full $2$-coloring of $\emb{\mathbf{B},\mathbf{K}}$ defined by $\delta(x) = \gamma(x\circ f)$. For some $i$, $\delta_i$ is thick. Then by Lemma 4.2, $\gamma_i = \{x\circ f: x\in \delta_i\}$ is thick.
\end{proof}
\vspace{2 mm}

We say that $\mathbf{A}\in \mathcal{C}$ has \emph{Ramsey degree} $k$ if $k$ is least such that for any $\mathbf{B}$ in $\mathcal{C}$ with $\mathbf{A}\leq \mathbf{B}$ and any $r > k$, there is $\mathbf{C}\in \mathcal{C}$ such that for any $r$-coloring $\gamma$ of $\emb{\mathbf{A},\mathbf{C}}$, there is $f\in \emb{\mathbf{B}, \mathbf{C}}$ such that $|\gamma(f\circ \emb{\mathbf{A},\mathbf{B}})|\leq k$. One could define the notion of an $(r,k)$-Ramsey object, which would be defined just as above for some particular $r>k$. However, this is unnecessary; an induction on $r$ shows that $\mathbf{A}$ is an $(r,k)$-Ramsey object iff $\mathbf{A}$ is a $(k+1,k)$-Ramsey object. Therefore the notion of Ramsey degree is sufficient. We use a similar hook-arrow notation,
$$\mathbf{C}\hookrightarrow (\mathbf{B})^\mathbf{A}_{r,k}$$
to mean that for every $r$-coloring $\gamma$ of $\emb{\mathbf{A},\mathbf{C}}$, there is $f\in \emb{\mathbf{B},\mathbf{C}}$ with $|\gamma(f\circ\emb{\mathbf{A},\mathbf{B}})|\leq k$. We use the standard arrow notation,
$$\mathbf{C}\rightarrow (\mathbf{B})^\mathbf{A}_{r,k}$$
to mean that for every $r$-coloring $\gamma$ of $\binom{\mathbf{C}}{\mathbf{A}}$, there is $\mathbf{B}_0\in\binom{\mathbf{C}}{\mathbf{B}}$ with $|\gamma(\binom{\mathbf{B}_0}{\mathbf{A}})|\leq k$.

\begin{prop}
$\mathbf{A}\in\mathcal{C}$ has structural Ramsey degree $k$ iff $\mathbf{A}$ has embedding Ramsey degree $k\cdot |\aut{\mathbf{A}}|$. 
\end{prop}

\begin{proof}
Set $t = |\aut{\mathbf{A}}|$, and fix $r>kt$. Assume $\mathbf{A}$ has structural-Ramsey degree $k$, and let $\mathbf{B}\in \mathcal{C}$ with $\mathbf{A}\leq \mathbf{B}$. Find $\mathbf{C}\in \mathcal{C}$ with $\mathbf{C}\rightarrow (\mathbf{B})^\mathbf{A}_{2^r, k}$. Fix any $r$-coloring $\gamma$ of $\emb{\mathbf{A},\mathbf{C}}$. Fix a bijection $\phi: \mathcal{P}([r])\rightarrow [2^r]$, and define a $(2^r)$-coloring $\delta$ with $\delta(f) =\phi(\{i\in [r]: \exists h\in \aut{\mathbf{A}}(f\circ h\in \gamma_i)\})$. Then $\delta$ is a structural coloring, so find $f\in \emb{\mathbf{B},\mathbf{C}}$ with $f\circ\emb{\mathbf{A},\mathbf{B}}$ at most $k$-colored for $\delta$. Then $f\circ\emb{\mathbf{A}, \mathbf{B}}$ is at most $(kt)$-colored for $\gamma$.
\vspace{2 mm}

Now since $\mathbf{A}$ has structural-Ramsey degree $k$, find $\mathbf{D}\in\mathcal{C}$ such that for every $\mathbf{E}\in\mathcal{C}$, there is a structural $r$-coloring $\gamma_\mathbf{E}$ of $\emb{\mathbf{A},\mathbf{E}}$ where for every $f\in\emb{\mathbf{D},\mathbf{E}}$, the set $f\circ \emb{\mathbf{A},\mathbf{D}}$ is at least $k$-colored. Then one can use $\gamma_\mathbf{E}$ to find an $(rt)$-coloring $\delta_\mathbf{E}$ of $\emb{\mathbf{A},\mathbf{E}}$ where each $f\circ \emb{\mathbf{A},\mathbf{D}}$ is at least $(kt)$-colored. One way to do this is as follows: first fix a bijection $\rho: \aut{\mathbf{A}}\rightarrow [t]$. Form the equivalence relation $\sim$ on $\emb{\mathbf{A},\mathbf{E}}$ where $f\sim g$ iff $f = g\circ h$ for some $h\in \aut{\mathbf{A}}$. Let $\langle f\rangle$ denote the equivalence class of $f$. Let $\psi: (\emb{\mathbf{A},\mathbf{E}}/\sim) \rightarrow \emb{\mathbf{A},\mathbf{E}}$ choose a member of each equivalence class. Now if $f\in \emb{\mathbf{A},\mathbf{E}}$ and $h\in \aut{\mathbf{A}}$ is the unique automorphism with $\psi(\langle f\rangle)\circ h = f$, then set $\delta_\mathbf{E}(f) = (\gamma_\mathbf{E}(f) -1)t+\rho(h)$.

Conversely, if $\mathbf{A}$ has finite embedding-Ramsey degree, then $\mathbf{A}$ also has finite structural-Ramsey degree, completing the proof.      
\end{proof}

\begin{cor}
$\mathbf{A}\in\mathcal{C}$ is an embedding Ramsey object iff $\mathbf{A}$ is a structural Ramsey object and is rigid.
\end{cor}

The proof of the following proposition is nearly identical to the proof of Proposition 4.1 and is therefore omitted:

\begin{prop}
Suppose $\mathbf{D}$ is a countably infinite locally finite structure, $\mathcal{D} = \mathrm{Age}(\mathbf{D})$, and $\mathcal{C}$ is cofinal in $\mathcal{D}$. Let $\mathbf{A}\in\mathcal{C}$, and fix $r > k$. Then the following are equivalent:
\begin{enumerate}
\item
$\mathbf{A}$ has Ramsey degree $t\leq k$ in $\mathcal{C}$,
\item
$\mathbf{A}$ has Ramsey degree $t\leq k$ in $\mathcal{D}$, 
\item
Any full $r$-coloring of $\emb{\mathbf{A},\mathbf{D}}$ has some subset of $k$ or fewer colors which form a thick subset,
\item
Any large $r$-coloring of $\emb{\mathbf{A},\mathbf{D}}$ has some subset of $k$ or fewer colors which form a thick subset.
\end{enumerate}
\end{prop}

There is a similar analogue to Proposition 4.3, which we also state without proof.

\begin{prop}
Suppose $\mathbf{K}$ is a \fr structure, $\mathcal{K} = \age{\mathbf{K}}$, and suppose $\mathbf{B}\in \mathcal{K}$ has Ramsey degree $k$. Then if $\mathbf{A}\leq \mathbf{B}$, then $\mathbf{A}$ has Ramsey degree $t\leq k$.
\end{prop}

Using Proposition 4.6, we can provide another definition of Ramsey degree which will be extremely useful going forward. Let $\mathbf{D}$ be a countably infinite locally finite structure with $\mathcal{D} = \age{\mathbf{D}}$, and let $\mathbf{A}\in\mathcal{D}$. Call a subset $S\subseteq \emb{\mathbf{A},\mathbf{D}}$ \emph{syndetic} if $S\cap X\neq \emptyset$ for every thick $X\subseteq \emb{\mathbf{A},\mathbf{D}}$. Call a coloring $\gamma$ of $\emb{\mathbf{A},\mathbf{D}}$ a \emph{syndetic coloring} if each $\gamma_i$ is syndetic. Now we have:

\begin{prop}
$\mathbf{A}$ has Ramsey degree $t\geq k$ \rm(\emph{$t$ possibly infinite}\rm) \emph{iff there is a syndetic $k$-coloring of $\emb{\mathbf{A},\mathbf{D}}$}.
\end{prop}

The words \emph{thick} and \emph{syndetic} are borrowed from topological dynamics. We will justify these vocabulary choices later (see the discussion after the proof of Proposition 8.1).

\section{KPT Correspondence}

We have now developed enough background to state the results in [KPT]. Our discussion will have two notable differences however. First we will be using embedding Ramsey throughout. Second, we will develop the theory using Fra\"iss\'e--HP classes, as this will allow us more flexibility in section 8. Later, we will provide new proofs of Theorems 5.1 and 5.7 (see Theorem 7.3 and the discussion after Corollary 8.15). 

\begin{theorem}
If $\mathcal{K}$ is a Fra\"iss\'e--HP class, $\mathbf{K} = \flim{\mathcal{K}}$, and $G = \aut{\mathbf{K}}$, then $M(G)$ is a singleton iff $\mathcal{K}$ has the Ramsey Property. 
\end{theorem}

\begin{rem}
Once again, we are using Ramsey Property for embeddings; see Corollary 4.5.
\end{rem}

Let $L$ be a language and $L^* = L\cup \mathcal{S}$, where $\mathcal{S} = \{S_i: i\in\mathbb{N}\}$ and the $S_i$ are new relational symbols of arity $n(i)$. If $\mathbf{A}$ is an $L^*$-structure, write $\mathbf{A}|_L$ for the structure obtained by throwing away the interpretations of the $S_i$. If $\mathcal{K}^*$ is a class of $L^*$-structures, set $\mathcal{K}^*|_L = \{\mathbf{A}^*|_L: \mathbf{A}^*\in \mathcal{K}^*\}$. If $\mathcal{K} = \mathcal{K}^*|_L$ and $\mathcal{K}^*$ is closed under isomorphism, we say that $\mathcal{K}^*$ is an \emph{expansion} of $\mathcal{K}$. If $\mathbf{A}^*\in\mathcal{K}^*$ and $\mathbf{A}^*|_L = \mathbf{A}$, then we say that $\mathbf{A}^*$ is an expansion of $\mathbf{A}$. If $f \in \emb{\mathbf{A}, \mathbf{B}}$ and $\mathbf{B}^*$ is an expansion of $\mathbf{B}$, we let $\mathbf{A}(f, \mathbf{B}^*)$ be the unique expansion of $\mathbf{A}$ so that $f\in \emb{\mathbf{A}(f, \mathbf{B}^*), \mathbf{B}^*}$. The expansion $\mathcal{K}^*$ is \emph{precompact} if for each $\mathbf{A}\in \mathcal{K}$, the set $\{\mathbf{A}^*\in\mathcal{K}^*: \mathbf{A}^*|_L = \mathbf{A}\}$ is finite. 

If $\mathcal{K}^*$ is an expansion of the class $\mathcal{K}$, it will be useful to think of the pair $(\mathcal{K}^*,\mathcal{K})$ as a category as follows. If $X\subseteq \mathcal{K}$ is a set (as opposed to a proper class), say that $X$ is \emph{adequate} for $\mathcal{K}$ if $X$ contains at least one representative of each isomorphism type in $\mathcal{K}$. For $X$ adequate, let $\mathrm{Cat}_X(\mathcal{K}^*,\mathcal{K})$ be the category $C$ with $\ob{C} = \{\mathbf{A}^*: \mathbf{A}^* \text{ is an expansion of some } \mathbf{A}\in X\}$ and with $\arr{C}$ the set of embeddings between structures in $\ob{C}$. If $\mathcal{K}^*$ and $\mathcal{K}^{**}$ are two expansions of the class $\mathcal{K}$ in languages $L^*$ and $L^{**}$, we say that $\mathcal{K}^*$ and $\mathcal{K}^{**}$ are \emph{isomorphic expansions} if for any adequate $X$, there is a fully faithful functor $\Phi_X: \mathrm{Cat}_X(\mathcal{K}^*,\mathcal{K})\rightarrow \mathrm{Cat}_X(\mathcal{K}^{**}, \mathcal{K})$ with fully faithful inverse satisfying
\begin{itemize}
\item
$\Phi_X(\mathbf{A}^*)|_L = \mathbf{A}^*|_L$ for any $\mathbf{A}^*\in\ob{\mathrm{Cat}_X(\mathcal{K}^*,\mathcal{K})}$,
\item
$\Phi_X(f) = f$ for any $f\in \arr{\mathrm{Cat}(\mathcal{K}^*,\mathcal{K})}$.
\end{itemize}
We will call such a $\Phi_X$ an \emph{isomorphism of expansions}. Notice that $L^*$ need not equal $L^{**}$ for $\mathcal{K}^*$ and $\mathcal{K}^{**}$ to be isomorphic.

\begin{prop}
If $\mathcal{K}^*$ and $\mathcal{K}^{**}$ are isomorphic, then for any adequate $X$, there is an isomorphism of expansions $\Phi_X: \mathrm{Cat}_X(\mathcal{K}^*,\mathcal{K})\rightarrow \mathrm{Cat}_X(\mathcal{K}^{**},\mathcal{K})$. 
\end{prop}

\begin{proof}
Let $X$ and $Y$ be adequate, and suppose that there is an isomorphism of expansions $\Phi_X$. Since isomorphisms of expansions are trivial on embeddings, it is enough to define $\Phi_Y$ on objects. For $\mathbf{A}\in Y$, choose isomorphic  $\mathbf{A}_X$ and an isomorphism $f_\mathbf{A}: \mathbf{A}_X\rightarrow \mathbf{A}$. Let $\mathbf{A}^*$ be a  $\mathcal{K}^*$-expansion of $\mathbf{A}$. Then $\Phi_X(\mathbf{A}_X(f_\mathbf{A},\mathbf{A}^*))$ is a $\mathcal{K}^{**}$-expansion of $\mathbf{A}_X$. Set $\Phi_Y(\mathbf{A}^*) = \mathbf{A}(f^{-1}_{\mathbf{A}}, \Phi_X(\mathbf{A}_X(f_\mathbf{A}, \mathbf{A}^*)))$. It is straightforward to check that this works.
\end{proof}

If $\mathcal{K}^*$ is an expansion of the Fra\"iss\'e--HP class $\mathcal{K}$, we say that the pair $(\mathcal{K}^*, \mathcal{K})$ is \emph{reasonable} if for any $\mathbf{A},\mathbf{B}\in\mathcal{K}$, embedding $f: \mathbf{A}\rightarrow \mathbf{B}$, and expansion $\mathbf{A}^*$ of $\mathbf{A}$, then there is an expansion $\mathbf{B}^*$ of $\mathbf{B}$ with $f: \mathbf{A}^*\rightarrow \mathbf{B}^*$ an embedding. When $\mathcal{K}^*$ is also a Fra\"iss\'e--HP class, we have the following equivalent definition.

\begin{prop}
Let $\mathcal{K}^*$ be a Fra\"iss\'e--HP expansion class of the Fra\"iss\'e--HP class $\mathcal{K}$ with \fr limits $\mathbf{K}^*,\mathbf{K}$ respectively. Then the pair $(\mathcal{K}^*,\mathcal{K})$ is reasonable iff $\mathbf{K}^*|_L \cong \mathbf{K}$.
\end{prop}

\begin{proof}
Assume that $\mathbf{K}^*|_L\cong \mathbf{K}$. Fix $\mathbf{A}^*\in\fin{\mathbf{K}^*}\cap \mathcal{K}^*$ and $\mathbf{B}\in \mathcal{K}$ with $\mathbf{A}^*|_L\subseteq \mathbf{B}$. Use the extension property for $\mathbf{K}$ to find an embedding $f: \mathbf{B}\rightarrow\mathbf{K}^*|_L$ extending the inclusion of $\mathbf{A}^*|_L$. Now define the expansion $\mathbf{B}^*$ of $\mathbf{B}$ by declaring that $S_i^{\mathbf{B}^*}(x_1,...,x_{n_i})$ holds iff $S_i^{\mathbf{K}^*}(f(x_1),...,f(x_{n_i}))$ holds. This is enough to show that $(\mathcal{K}^*,\mathcal{K})$ is reasonable.

Conversely, suppose $(\mathcal{K}^*,\mathcal{K})$ is reasonable. Let $\mathbf{K}^* = \bigcup_{n} \mathbf{A}_n^*$ be a $\mathcal{K}^*$-exhaustion, and let $\mathbf{K} = \bigcup_n \mathbf{B}_n$ be a $\mathcal{K}$-exhaustion. Set $\mathbf{A}_n = \mathbf{A}^*_n|_L$; let $f_1: \mathbf{A}_1\rightarrow \mathbf{B}_{n_1}$ for some large enough $n_1$. Using the reasonable property, choose an expansion $\mathbf{B}_{n_1}^*$ of $\mathbf{B}_{n_1}$ with $f_1: \mathbf{A}_1^*\rightarrow \mathbf{B}_{n_1}^*$ an embedding. Then use the extension property for $\mathbf{K}^*$ to find an embedding $f_2: \mathbf{B}_{n_1}\rightarrow \mathbf{A}_{n_2}$ extending $f_1^{-1}$ for some large enough $n_2$. If $f_k$ is defined and $k$ is even, use the extension property for $\mathbf{K}$ to find $f_{k+1}: \mathbf{A}_{n_k}\rightarrow \mathbf{B}_{n_{k+1}}$ extending $f_k^{-1}$. If $k$ is odd, use the reasonable property and the extension property for $\mathbf{K}^*$ to define $f_{k+1}$ extending $f_k^{-1}$. We proceed in this manner, building an isomorphism $\bigcup_{n} f_{2n}: \mathbf{K}\rightarrow \mathbf{K}^*|_L$.
\end{proof}
\vspace{2 mm}

Now suppose $(\mathcal{K}^*,\mathcal{K})$ is reasonable and precompact. Set
$$X_{\mathcal{K}^*} := \{\langle \mathbf{K},\vec{S}\rangle: \langle \mathbf{A}, \vec{S}|_A\rangle \in \mathcal{K}^* \text{ whenever } \mathbf{A}\in \fin{\mathbf{K}}\cap\mathcal{K}\}$$
We topologize this space by declaring the basic open neighborhoods to be of the form $N(\mathbf{A}^*) := \{\mathbf{K}'\in X_\mathcal{K}^*: \mathbf{A}^*\subseteq \mathbf{K}'\}$, where $\mathbf{A}^*$ is an expansion of some $\mathbf{A}\in\fin{\mathbf{K}}\cap \mathcal{K}$. We can view $X_{\mathcal{K}^*}$ as a closed subspace of 
$$\prod_{\mathbf{A}\in \fin{\mathbf{K}}\cap \mathcal{K}} \{\mathbf{A}^*: \mathbf{A}^* \text{ is an expansion of } \mathbf{A}\}$$.
Notice that since $(\mathcal{K}^*,\mathcal{K})$ is precompact, $X_{\mathcal{K}^*}$ is compact. If $\bigcup_n \mathbf{A}_n = \mathbf{K}$ is an exhaustion, a compatible metric is given by 
$$d(\langle \mathbf{K}, \vec{S}\rangle, \langle \mathbf{K}, \vec{T}\rangle = 1/k(\vec{S}, \vec{T})$$
where $k(\vec{S}, \vec{T})$ is the largest $k$ for which $\langle \mathbf{A}_k, \vec{S}|_{A_k}\rangle\cong \langle \mathbf{A}_k, \vec{T}|_{A_k}\rangle$.

We can now form the (right) logic action of $G = \aut{\mathbf{K}}$ on $X_{\mathcal{K}^*}$ by setting $\mathbf{K}'\cdot g$ to be the structure where $S_i^{(\mathbf{K}'g)}(x_1,...,x_{n_i})$ holds iff $S_i^{\mathbf{K}'}(g(x_1),...,g(x_{n_i}))$ holds. It is easy to check that this action is jointly continuous, turning $X_{\mathcal{K}^*}$ into a $G$-flow. For readers used to left logic actions, acting on the right by $g$ is the same as acting on the left by $g^{-1}$. 

\begin{prop}
If $\mathcal{K}$ is a Fra\"iss\'e--HP class and $\mathcal{K}^*$ and $\mathcal{K}^{**}$ are isomorphic expansions of $\mathcal{K}$ with each reasonable and precompact, then $X_{\mathcal{K}^*}\cong X_{\mathcal{K}^{**}}$. 
\end{prop}

\begin{proof}
Let $\mathbf{K} = \flim{\mathcal{K}}$, and fix a $\mathcal{K}$-exhaustion $\bigcup_n \mathbf{A}_n$. Set $X = \fin{\mathbf{K}}\cap \mathcal{K}$, and let $\Phi_X: \mathrm{Cat}_X(\mathcal{K}^*,\mathcal{K})\rightarrow \mathrm{Cat}_X(\mathcal{K}^{**},\mathcal{K})$ be an isomorphism of expansions. Define a map $\phi: X_{\mathcal{K}^*}\rightarrow X_{\mathcal{K}^{**}}$ via $\phi(\langle \mathbf{K}, \vec{S}\rangle) = \bigcup_n \Phi_X(\langle \mathbf{A}_n, \vec{S}|_{A_n}\rangle)$. Notice that since $\Phi_X$ respects embeddings, the right hand side is a member of $X_{\mathcal{K}^{**}}$. It is straightforward to check that this is a continuous bijection which respects $G$-action.
\end{proof}
\vspace{2 mm}

First let us consider when $X_{\mathcal{K}^*}$ is a minimal $G$-flow. We say that the pair $(\mathcal{K}^*,\mathcal{K})$ has the \emph{Expansion Property} (ExpP) when for any $\mathbf{A}^*\in\mathcal{K}^*$, there is $\mathbf{B}\in\mathcal{K}$ such that for any expansion $\mathbf{B}^*$ of $\mathbf{B}$, there is an embedding $f: \mathbf{A}^*\rightarrow \mathbf{B}^*$.

\begin{prop}
Let $\mathcal{K}^*$ be a reasonable, precompact Fra\"iss\'e--HP expansion class of the Fra\"iss\'e--HP class $\mathcal{K}$ with \fr limits $\mathbf{K}^*, \mathbf{K}$ respectively. Let $G = \aut{\mathbf{K}}$. Then the $G$-flow $X_{\mathcal{K}^*}$ is minimal iff the pair $(\mathcal{K}^*,\mathcal{K})$ has the ExpP.
\end{prop}

\begin{proof}
Suppose the pair has ExpP. Let $\mathbf{A}^*\in\mathcal{K}^*$, and find $\mathbf{B}\in\mathcal{K}$ witnessing the ExpP for $\mathbf{A}^*$. Pick any $\mathbf{K}'\in X_{\mathcal{K}^*}$, and find $\mathbf{B}'\subseteq \mathbf{K}'$ with $\mathbf{B}'|_L = \mathbf{B}$. Then there is an embedding $f: \mathbf{A}^*\rightarrow \mathbf{B}'\subseteq \mathbf{K}'$, hence $\mathcal{K}^*\subseteq \age{\mathbf{K}'}$. Now fix a basic open neighborhood $N(\mathbf{C}^*)$ of $X_{\mathcal{K}^*}$, where $\mathbf{C}^*$ expands $\mathbf{C}\in\fin{\mathbf{K}}\cap \mathcal{K}$.  Let $f: \mathbf{C}^*\rightarrow \mathbf{D}'$ be an embedding. Use ultrahomogeneity in $\mathbf{K}$ to find $g\in G$ extending $f$. Then $\mathbf{K}'\cdot g\in N(\mathbf{C}^*)$; hence $X_{\mathcal{K}^*}$ is minimal.

Conversely, suppose the pair does not have ExpP. Find $\mathbf{A}^*\in\fin{\mathbf{K}^*}\cap \mathcal{K}^*$ for which there is no $\mathbf{B}\in\mathcal{K}$ witnessing ExpP. Now use K\"onig's Lemma to find a $\mathbf{K}'\in X_{\mathcal{K}^*}$ with $\mathbf{A}^*\not\in\age{\mathbf{K}'}$.  It follows that $\mathbf{K}'\cdot G\cap N(\mathbf{A}^*) = \emptyset$, so $X_{\mathcal{K}^*}$ is not minimal.
\end{proof}

\begin{cor}
With the assumptions of Proposition 5.5, suppose $(\mathcal{K}^*, \mathcal{K})$ has ExpP. Let $\mathbf{A}\in \mathcal{K}$, and let $\mathbf{A}^*$ be an expansion of $\mathbf{A}$. Then $\emb{\mathbf{A}^*,\mathbf{K}^*}$ is a syndetic subset of $\emb{\mathbf{A},\mathbf{K}}$.
\end{cor}

\begin{proof}
Let $\mathbf{B}\in\mathcal{K}$ witness the ExpP for $\mathbf{A}^*$. If $X\subseteq \emb{\mathbf{A}, \mathbf{K}}$ is thick, find $g\in \emb{\mathbf{B},\mathbf{K}}$ with $g\circ\emb{\mathbf{A},\mathbf{B}}\subseteq X$. Let $\mathbf{B}^* = \mathbf{B}(g, \mathbf{K}^*)$. Since $\mathbf{B}$ witnesses ExpP for $\mathbf{A}^*$, find $f\in \emb{\mathbf{A}^*,\mathbf{B}^*}$. Now $g\circ f\in \emb{\mathbf{A}^*,\mathbf{K}^*}\cap X$.
\end{proof}
\vspace{2 mm}

The following extends Theorem 5.1 and is one of the major theorems in [KPT]. This theorem in its full generality is proven in [NVT]:

\begin{theorem}
Let $\mathcal{K}^*$ be a reasonable, precompact Fra\"iss\'e--HP expansion class of the Fra\"iss\'e--HP class $\mathcal{K}$ with \fr limits $\mathbf{K}^*,\mathbf{K}$, respectively. Let $G = \aut{\mathbf{K}}$. Then $X_{\mathcal{K}^*}\cong M(G)$ iff the pair $(\mathcal{K}^*,\mathcal{K})$ has the ExpP and $\mathcal{K}^*$ has the RP.
\end{theorem}
\begin{rem}
Pairs $(\mathcal{K}^*,\mathcal{K})$ of Fra\"iss\'e--HP classes which are reasonable, precompact, and satisfy the ExpP and RP are called \emph{excellent}.
\end{rem}

When $(\mathcal{K}^*, \mathcal{K})$ is an excellent pair, the Ramsey Property for $\mathcal{K}^*$ tells us something about the combinatorics of $\mathcal{K}$

\begin{prop}
Let $(\mathcal{K}^*,\mathcal{K})$ be an excellent pair. Then every $\mathbf{A}\in \mathcal{K}$ has finite Ramsey degree. In particular, the Ramsey degree of $\mathbf{A}$ is equal to the number of expansions of $\mathbf{A}\in \mathcal{K}^*$. 
\end{prop}

Let's prove this by using some of the ideas developed in section 4.

\begin{lemma}
With $(\mathcal{K}^*, \mathcal{K})$ as in Proposition 5.8, let $\mathbf{A}\in \mathcal{K}$ and $\mathbf{A}^*$ be an expansion of $\mathbf{A}$. Then $X\subseteq \emb{\mathbf{A}^*, \mathbf{K}^*}$ is thick iff $X = T\cap \emb{\mathbf{A}^*,\mathbf{K}^*}$ with $T\subseteq \emb{\mathbf{A},\mathbf{K}}$ thick.
\end{lemma}

\begin{rem}
``Thick'' above is referring to two different notions of thickness. In general, when we say $X\subseteq \emb{\mathbf{A},\mathbf{D}}$ is thick/syndetic, this means with respect to the class $\mathcal{D} = \age{\mathbf{D}}$.
\end{rem}

\begin{proof}
$(\Rightarrow)$ If $X\subseteq \emb{\mathbf{A}^*,\mathbf{K}^*}$ is thick, fix $\mathbf{B}\in \mathcal{K}$ with $\mathbf{A}\leq \mathbf{B}$. Pick any expansion $\mathbf{B}^*$ of $\mathbf{B}$, then as $X\subseteq \emb{\mathbf{A}^*,\mathbf{K}^*}$ is thick, find $g_\mathbf{B}\in \emb{\mathbf{B}^*,\mathbf{K}^*}$ with $g_\mathbf{B}\circ \emb{\mathbf{A}^*,\mathbf{B}^*}\subseteq X$. Now set $T' = \bigcup_{\mathbf{A}\leq\mathbf{B}} g_\mathbf{B}\circ \emb{\mathbf{A},\mathbf{B}}$. Then $T'\subseteq \emb{\mathbf{A},\mathbf{K}}$ is thick, as is $T = T'\cup X$. Then $X = T\cap \emb{\mathbf{A}^*,\mathbf{K}^*}$ as desired.

$(\Leftarrow)$ If $X = T\cap \emb{\mathbf{A}^*, \mathbf{K}^*}$ with $T\subseteq \emb{\mathbf{A},\mathbf{K}}$ thick, then fix $\mathbf{B}^*\in\mathcal{K}^*$ with $\mathbf{A}^*\leq \mathbf{B}^*$. Then find $\mathbf{C}\in\mathcal{K}$ witnessing the ExpP for $\mathbf{B}^*$. As $T\subseteq \emb{\mathbf{A},\mathbf{K}}$ is thick, find $g\in\emb{\mathbf{C},\mathbf{K}}$ with $g\circ \emb{\mathbf{A},\mathbf{C}}\subseteq T$. Let $\mathbf{C}^* = \mathbf{C}(g, \mathbf{K}^*)$ be the unique expansion of $\mathbf{C}$ so that $g\in\emb{\mathbf{C}^*,\mathbf{K}^*}$. As $\mathbf{C}$ witnesses ExpP for $\mathbf{B}^*$, pick $f\in \emb{\mathbf{B}^*,\mathbf{C}^*}$. Now we have $g\circ f\in \emb{\mathbf{B}^*,\mathbf{K}^*}$ and $g\circ f\circ \emb{\mathbf{A}^*,\mathbf{B}^*}\subseteq T\cap \emb{\mathbf{A}^*,\mathbf{K}^*}$.
\end{proof}

\begin{proof}[Proof of Proposition 5.8]
Fix $\mathbf{A}\in \mathcal{K}$, and let $\mathbf{A}_1,...,\mathbf{A}_k$ list the expansions of $\mathbf{A}$. We can now write
$$\emb{\mathbf{A},\mathbf{K}} = \bigsqcup_{i\leq k} \emb{\mathbf{A}_i, \mathbf{K}^*}.$$
Fix a $k+1$-coloring $\gamma$ of $\emb{\mathbf{A},\mathbf{K}}$. Find a thick $T_1\subseteq \emb{\mathbf{A},\mathbf{K}}$ so that $T_1\cap \emb{\mathbf{A}_1,\mathbf{K}^*}$ is monochromatic. If thick $T_i\subseteq \emb{\mathbf{A},\mathbf{K}}$ has been determined, find thick $T_{i+1}\subseteq \emb{\mathbf{A},\mathbf{K}}$, $T_{i+1}\subseteq T_i$, so that $T_{i+1}\cap \emb{\mathbf{A}_{i+1}, \mathbf{K}^*}$ is monochromatic. Then $T_k\subseteq \emb{\mathbf{A},\mathbf{K}}$ is thick and $|\gamma(T_k)|\leq k$. This shows that $\mathbf{A}$ has Ramsey degree $\leq k$.

For the other bound, note that by Corollary 5.6, $\emb{\mathbf{A}_i, \mathbf{K}^*}$ is syndetic. Let $\gamma$ be the coloring  of $\emb{\mathbf{A},\mathbf{K}}$ with $\gamma(f) = i$ iff $f\in \emb{\mathbf{A}_i, \mathbf{K}^*}$. Then $\gamma$ is a syndetic $k$-coloring, so by Proposition 4.8, $\mathbf{A}$ has Ramsey degree $\geq k$.
\end{proof}

In the setting of Theorem 5.7, we can consider the orbit of $\mathbf{K}^*$ in $X_{\mathcal{K}^*} = M(G)$. Notice that $\mathbf{K}' = \mathbf{K}^*\cdot g$ iff $\mathbf{K}'\cong \mathbf{K}^*$ iff $\mathbf{K}'$ has age $\mathcal{K}^*$ and satisfies the extension property. We can write out which $\mathbf{K}'$ satisfy these assumptions as a countable intersection of open conditions in $X_{\mathcal{K}^*}$. Since $X_{\mathcal{K}^*}$ is minimal, each $\mathbf{K}'$ has age $\mathcal{K}^*$.  $\mathbf{K}'\in X_{\mathcal{K}^*}$ satisfies the extension property iff it is in each of the open neighborhoods $N(\mathbf{A}^*\subseteq \mathbf{B}^*)$, where $\mathbf{B}^*\in\mathcal{K}^*$ and $\mathbf{A}^*$ is an expansion of $\mathbf{A}\in\fin{\mathbf{K}}\cap \mathcal{K}$ (recall the notion of isomorphic inclusion pair defined before Proposition 3.3. Note that since $\mathcal{K}^*$ has RP, $\mathbf{A}^*$ is rigid):
$$N(\mathbf{A}^*\subseteq \mathbf{B}^*) := \left(\bigcup_{\substack{\mathbf{A}'|_L = \mathbf{A}\\ \mathbf{A}'\neq \mathbf{A}^*}} N(\mathbf{A}')\right) \cup \left( \bigcup_{(\mathbf{A}^*\subseteq \mathbf{B}')\cong (\mathbf{A}^*\subseteq \mathbf{B}^*)} N(\mathbf{B}')\right).$$
The orbit $\mathbf{K}^*\cdot G$ is also dense since $X_{\mathcal{K}^*}$ is minimal; hence $\mathbf{K}^*\cdot G$ is a generic orbit in $X_{\mathcal{K}^*}$. Note that any $G$-flow can have at most one generic orbit as the intersection of two generic subsets of any Baire space is nonempty. The following proposition is proved in [AKL] (Prop.\ 14.1), we follow that proof for the most part.

\begin{prop}
Let $G$ be a Polish topological group and suppose $M(G)$ has a generic orbit. Then if $Y$ is a minimal $G$-flow, then $Y$ has a generic orbit. 
\end{prop}

\begin{proof}
Let $\pi: M(G)\rightarrow Y$ be a $G$-map, and suppose $x_0\in M(G)$ has a generic orbit. We will show that $y_0 := \pi(x_0)$ has a generic orbit in $Y$. First, we show that $y_0\cdot G$ has the Baire Property (BP). Fix a continuous surjection $f: \mathcal{N}\rightarrow G$, where $\mathcal{N}$ is the Baire space. For $s\in {}^{< \omega}\mathbb{N}$ write $G_s = \{f(\alpha): \alpha = s\char94 \beta\}$. Set $P_s = \overline{y_0\cdot G_s}$. Now $\{P_s: s\in {}^{< \omega}\mathbb{N}\}$ is a regular Souslin scheme. Set 
$$P = \bigcup_{\alpha\in \mathcal{N}} \bigcap_{n\in \mathbb{N}} P_{\alpha |_n}.$$
Since each $P_s$ is closed, hence has the BP, the Nikod\'{y}m Theorem (see [K], 29.14) tells us that $P$ also has the BP, and $P = y_0\cdot G$.

Suppose for sake of contradiction that $y_0\cdot G$ is not comeager, then as $y_0\cdot G$ is $G$-invariant and $Y$ is minimal, $y_0\cdot G$ must be meager. Fix dense open $V_n\subseteq Y$ with $(y_0\cdot G)\cap (\bigcap_n V_n)$ empty. Taking preimages, we must have $(x_0\cdot G)\cap \left(\bigcap_n \pi^{-1}(V_n)\right) = \emptyset$. Set $U_n = \pi^{-1}(V_n)$. We will derive a contradiction once we show that $U_n$ is dense open. Pick $W\subseteq X$ nonempty and open, and let $G_0\subseteq G$ be a countable dense subgroup. Since $M(G)$ is minimal, $W\cdot G_0 = X$. Then $\pi(W)\cdot G_0 = Y$, so $\pi(W)$ is not meager. Hence $\pi(W)\cap V_n\neq \emptyset$, and thus $W\cap U_n\neq \emptyset$.
\end{proof}

If $G$ is Polish and $M(G)$ has a generic orbit, $G$ is said to have the \emph{generic point property}. Angel, Kechris, and Lyons posed the following conjecture ([AKL] Question 15.2):

\begin{conj}[Generic Pont Problem]
Let $G$ be a Polish group with metrizable universal minimal flow. Then $G$ has the generic point property.
\end{conj}

Our new proof of Theorems 5.1 and 5.7 will have the added benefit of solving the Generic Point Problem for $G$ a closed subgroup of $S_\infty$.

\section{The Greatest Ambit}

In the remaining sections, we fix once and for all a relational \fr class $\mathcal{K}$ with \fr limit $\mathbf{K}$, which we suppose has universe $\mathbb{N}$. Set $G = \aut{\mathbf{K}}$. For each $n\in \mathbb{N}$ we also let $\mathbf{A}_n\in \fin{\mathbf{K}}$ with $A_n = \{1,2,...,n\}$. As a shorthand, write $H_n$ for $\emb{\mathbf{A}_n, \mathbf{K}}$; let $i_m^n$ denote the inclusion $\mathbf{A}_m\hookrightarrow \mathbf{A}_n$ for $m\leq n$ and $i_n$ denote the inclusion embedding $\mathbf{A}_n\subseteq \mathbf{K}$. 
\vspace{1.5 mm}

Suppose $f\in \emb{\mathbf{A}_m,\mathbf{A}_n}$. As $\mathbf{K}$ is a \fr structure, the map $\hat{f}: H_n\rightarrow H_m$ given by $\hat{f}(g) = g\circ f$ is surjective. Let $\beta H_n$ denote the $\beta$-compactification of the discrete space $H_n$. Then $\hat{f}$ has a unique continuous extension $\tilde{f}: \beta H_n\rightarrow \beta H_m$ which is also surjective. If $q\in \beta H_n$ and $S\subseteq H_m$, then $S\in \tilde{f}(q)$ iff $\hat{f}^{-1}(S)\in q$, i.e.\ $\tilde{f}(q)$ is just the pushforward of $q$ by $\hat{f}$. We will primarily be interested in the case when $f = i_m^n$.
\vspace{1.5 mm} 

Form the space $\varprojlim \beta H_n := \{\alpha\in \prod_n \beta H_n: \tilde{i}_m^n(\alpha(n)) = \alpha(m)\}$. Topologically, we view $\varprojlim \beta H_n$ as a subspace of $\prod_n \beta H_n$. Let $1\in \varprojlim \beta H_n$ denote the element where on each level $n$, the ultrafilter is principal on the embedding $1_G|_{\mathbf{A}_n}$. Our goal is to give $\varprojlim \beta H_n$ a $G$-flow structure; then $(\varprojlim \beta H_n, 1)$ will be the greatest $G$-ambit. To do this, we first take the peculiar step of stripping away the pointwise convergence topology on $G$, replacing it with the discrete topology. Form $\beta G$, which here will always refer to the compactification of $G$ as a discrete space. Endow $\beta G$ with the left-topological semigroup structure extending $G$. For the most part, we will only need the right $G$-action that arises from this structure; if $p\in \beta G$, $g\in G$, and $S\subseteq G$, then $S\in pg$ iff $Sg^{-1}\in p$
\vspace{1.5 mm}

Let $\tilde{\pi}_n: \beta G\rightarrow \beta H_n$ be the unique continuous extension of the map $\pi_n(g) = g|_{\mathbf{A}_n}$, and let $\tilde{\pi}: \beta G\rightarrow \varprojlim \beta H_n$ be given by $(\tilde{\pi}(p))(n) = \tilde{\pi}_n(p)$. Implicit in this definition is that $\tilde{\pi}_m  = \tilde{i}_m^n\circ \tilde{\pi}_n$, which follows since $\pi_m = \hat{i}_m^n\circ \pi_n$. Notice that $\tilde{\pi}$ is continuous, $1 = \tilde{\pi}(1_G)$, and $\{\tilde{\pi}(g): g\in G\}$ is dense, hence $\tilde{\pi}$ is surjective. We can use the semigroup structure on $\beta G$ to give $\varprojlim \beta H_n$ a $G$-action.

\begin{prop}
If $p,q\in \beta G$ are such that $\tilde{\pi}(p) = \tilde{\pi}(q)$, then $\tilde{\pi}(pg) = \tilde{\pi}(qg)$ for any $g\in G$.
\end{prop}

\begin{proof}
Fix $S\subseteq H_m$ and $g\in G$. Choose $n$ large enough so that $\mathbf{A}_m\cup g(\mathbf{A}_m)\subseteq \mathbf{A}_n$. Set $T = \{f\in H_n: f \circ g|_{\mathbf{A}_m}\in S\}$. Then we have:

\begin{align*}
\pi_m^{-1}(S)\in pg &\Leftrightarrow \{h\in G: hg\in \pi_m^{-1}(S)\}\in p\\
&\Leftrightarrow \{h\in G: hg|_{\mathbf{A}_m}\in S\}\in p\\
&\Leftrightarrow \{h\in G: h|_{\mathbf{A}_n}\in T\}\in p\\
&\Leftrightarrow \{h\in G: h|_{\mathbf{A}_n}\in T\}\in q\\
&\Leftrightarrow \pi_m^{-1}(S)\in qg \qedhere
\end{align*}
\end{proof}

We now can define $\tilde{\pi}(p)\cdot g := \tilde{\pi}(pg)$. That this is an action follows from associativity of $\beta G$. More explicitly, if $\alpha\in \varprojlim \beta H_n$, $g\in G$, and $S\subseteq H_m$, we have for large $n$ that
\begin{align*}
S\in \alpha g(m) \Leftrightarrow \{f\in H_n: f \circ g|_{\mathbf{A}_m}\in S\}\in \alpha (n).
\end{align*} 
Our use of right actions instead of left actions is justified by the following:

\begin{prop}
The right action of $G$ on $\varprojlim \beta H_n$ is jointly continuous when $G$ is given the pointwise convergence topology.
\end{prop}

\begin{proof}
First note that a basis for the topology on $\varprojlim \beta H_n$ is given by sets of the form $\tilde{S} := \{\alpha: S\in\alpha (m)\}$, where $S\subseteq H_m$ and $m\in \mathbb{N}$. So suppose $S\subseteq H_m$ and $\alpha g\in \tilde{S}$. Fix $n$ large enough so that $g(\mathbf{A}_m)\subseteq \mathbf{A}_n$, and let $T = \{f\in H_n: f\circ g|_{\mathbf{A}_m}\in S\}$. Letting $U_m = \{h\in G: h|_{\mathbf{A}_m} = g|_{\mathbf{A}_m}\}$, then $\tilde{T}\cdot U_m\subseteq \tilde{S}$.
\end{proof}

\begin{theorem}
$(\varprojlim \beta H_n, 1)$ is the greatest $G$-ambit when $G$ is given the pointwise convergence topology.
\end{theorem}

\begin{proof}
Let $(X, x_0)$ be a $G$-ambit, and let $\rho: \beta G\rightarrow X$ be the continuous extension of the map $g\rightarrow x_0\cdot g$ (remember that $\beta G$ is the compactification constructed from the discrete topology on $G$). Write $\rho(p) = x_0\cdot p$; notice that if $U\ni x_0\cdot p$ is an open neighborhood, then $\{g\in G: x_0\cdot g\in U\}\in p$. We will show that if $\tilde{\pi}(p) = \tilde{\pi}(q)$, then $x_0\cdot p = x_0\cdot q$. So suppose $x_0\cdot p \neq x_0\cdot q$. As compact Hausdorff spaces are normal, pick $U_p, V_p, U_q, V_q$ open neighborhoods of $x_0\cdot p, x_0\cdot q$ with $\overline{U_p}\subseteq V_p$, $\overline{U_q}\subseteq V_q$, and $V_p\cap V_q = \emptyset$. For each $y\in X\setminus V_p$, let $W_y^p$ be a neighborhood avoiding $U_p$. For each $n$, let $G_n\subseteq G$ denote the pointwise stabilizer of $\mathbf{A}_n$; use joint continuity to find $Y_y^p$ a neighborhood of $y$ and $n_y^p\in \mathbb{N}$ with $Y_y^p\cdot G_{n_y^p}\subseteq W_y^p$. As the $Y_y^p$ cover $X\setminus V_p$, find a finite subcover $\mathcal{C}_p$. Repeat these steps for $q$, and let $N$ be the largest of any $n_y^p, n_y^q$ mentioned in $\mathcal{C}_p, \mathcal{C}_q$. 
\vspace{1.5 mm}

Now if $x_0\cdot g\in U_p$, we must have $x_0\cdot (gG_N)\subseteq V_p$; if this were not the case, then for $h\in gG_N$ with $x_0\cdot h\in X\setminus V_p$, we have $x_0\cdot h\in Y_y^p$ for some $Y_y^p\in\mathcal{C}_p$. Hence $x_0\cdot (gG_N) = (x_0 h)G_N\subseteq W_y^p$, a contradiction since $W_y^p\cap U_p = \emptyset$. Similarly for $q$. Let $S_p := \{f\in H_N: \exists g\in G(g|_{\mathbf{A}_N} = f\text{ and } x_0\cdot g\in U_p)\}$. Do likewise for $q$. Then $\{g\in G: x_0\cdot g\in U_p\}\subseteq \pi_N^{-1}(S_p)\in p$, $\{g\in G: x_0\cdot g\in U_q\}\subseteq \pi_N^{-1}(S_q)\in q$, and $\pi_N^{-1}(S_p)\cap \pi_N^{-1}(S_q) = \emptyset$. Hence $\tilde{\pi}(p)\neq \tilde{\pi}(q)$.
\vspace{1.5 mm}

Thus there is a well-defined map $\phi: \varprojlim \beta H_n\rightarrow X$ with $\rho = \phi\circ \tilde{\pi}$. To show that $\phi$ is a $G$-map, we need to show that $\phi$ is continuous and respects $G$-action. Since $\tilde{\pi}$ is a continuous and closed map, we see that $\phi$ is continuous. For fixed $g\in G$, observe that $p\rightarrow x_0\cdot (pg)$ and $p\rightarrow (x_0\cdot p)\cdot g$ are two continuous extensions of $h\rightarrow x_0\cdot hg$, so are equal. Hence $\rho(pg) = \rho(p)\cdot g$ for any $p\in \beta G$, $g\in G$. Now let $\alpha\in \varprojlim \beta H_n$, $g\in G$, and pick $p\in \beta G$ with $\tilde{\pi}(p) = \alpha$. Then $\phi(\alpha\cdot g) = \phi(\tilde{\pi}(p)\cdot g) = \phi(\tilde{\pi}(pg)) = \rho(pg) = \rho(p)\cdot g = \phi(\alpha)\cdot g$. Hence $\phi$ is a $G$-map.
\end{proof}

As a first application, we easily obtain the following corollary, originally due to Pestov [P].

\begin{cor}
If $G$ is a closed subgroup of $S_\infty$ with infinite metrizable universal minimal flow $M(G)$, then as a topological space, $M(G)\cong 2^{\mathbb{N}}$.
\end{cor}

\begin{proof}
Let $Y\subseteq \varprojlim \beta H_n$ be infinite metrizable. Notice that if $Y$ is a subflow, then $Y$ has no isolated points. By Theorem 2.9, $\beta H_n$ embeds no infinite compact metric space. Consider the projection of $\varprojlim \beta H_n$ onto the $n$-th coordinate; by Proposition 2.10, it must be that for each $n$, there is a finite $Y_n\subseteq \beta H_n$ with $\alpha (n)\in Y_n$ for any $\alpha\in Y$. It follows that $Y = \varprojlim Y_n$.
\end{proof}

\begin{rem}
Lionel Nguyen Van Th\'e has pointed out to me that the construction in this section is essentially a more explicit version of the original construction of the greatest ambit $S(G)$ (of any topological group $G$) given by Pierre Samuel [Sa]. His construction proceeds as follows: let $\mathcal{V}$ be a basis of open neighborhoods of the identity. For $p\in \beta G$ (as a discrete group), let $p^*$ be the filter generated by sets of the form $\{SV: S\in p, V\in \mathcal{V}\}$. Now set $p\sim q$ iff $p^* = q^*$; we then obtain $S(G)\cong \beta G/\sim$. Dana Barto\v{s}ov\'a uses this approach to extend some of the results from [KPT] to uncountable structures (see [B]). The representation of the greatest ambit presented here was also discovered by Pestov (Corollary 3.3 in [P]).
\end{rem}

\section{Extreme Amenability}

We now turn to the proof of Theorem 5.1; though logically we could skip to the proof of Theorem 5.7, this will provide an introduction to many of the ideas used there. Since $M(G)$ is isomorphic to any minimal subflow of $\varprojlim \beta H_n$, it is enough to characterize when $\varprojlim \beta H_n$ has a fixed point. 

We say that an ultrafilter $p\in \beta H_n$ is \emph{thick} if each $S\in p$ is thick. Denote the set of thick ultrafilters on $H_n$ by $R_n$.
\begin{prop}
$R_n\neq \emptyset$ iff $\mathbf{A}_n$ is a Ramsey object in $\mathcal{K}$. 
\end{prop}

\begin{proof}
To see this, we need to show that the non-thick subsets of $H_n$ form an ideal iff $\mathbf{A}_n$ is a Ramsey object. If $\mathbf{A}_n$ is a Ramsey object, suppose $S\subseteq H_n$ is thick, and suppose $S = T_1\sqcup T_2$. By the equivalence of (1) and (4) in Proposition 4.1, we see that one of $T_1$ or $T_2$ is Ramsey. 

Conversely, if $\mathbf{A}_n$ is not a Ramsey object, then use the equivalence of (1) and (3) in Proposition 4.1 to find disjoint $S,T\subseteq H_n$ with $S\sqcup T = H_n$ and neither $S$ nor $T$ Ramsey.
\end{proof}

\begin{prop}
Suppose $m\leq n$, $\mathbf{A}_n$ is a Ramsey object, and $f\in \emb{\mathbf{A}_m, \mathbf{A}_n}$. Then if $p\in R_m$, there is $q\in R_n$ with $\tilde{f}(q) = p$.
\end{prop}

\begin{proof}

Form the preimage filter $\hat{f}^{-1}(p)$. If $T\in \hat{f}^{-1}(p)$, then $T \supseteq \hat{f}^{-1}(S)$ for some $S\in p$. Suppose $T$ is not thick; find a large enough $N$ so that for each $g\in H_N$, we have $g\circ \emb{\mathbf{A}_n, \mathbf{A}_N}\not\subseteq T$. It follows that for any $g\in H_N$, $g\circ \emb{\mathbf{A}_n, \mathbf{A}_N}\circ f\not\subseteq S$. As $\emb{\mathbf{A}_n, \mathbf{A}_N}\circ f\subseteq \emb{\mathbf{A}_m,\mathbf{A}_N}$, we see that $S$ is not thick, a contradiction. Now extend $\hat{f}^{-1}(p)$ to any $q\in \beta H_n$ avoiding the non-thick ideal.
\end{proof}

It follows from Propositions 7.1 and 7.2 that $\varprojlim R_n\neq \emptyset$ iff $\mathcal{K}$ has the Ramsey Property. It is natural to ask whether this is a subflow of $\varprojlim \beta H_n$; in fact, we can do even better. The following theorem implies Theorem 5.1.

\begin{theorem}
$\alpha\in \varprojlim \beta H_n$ is a fixed point iff $\alpha\in \varprojlim R_n$.
\end{theorem}

\begin{proof}
Suppose $\alpha\in \varprojlim R_n$, and let $S\in \alpha (m)$. Fix $g\in G$; we want to show that $S\in \alpha g(m)$. Let $n\geq m$ be large enough so that $\mathbf{A}_m\cup g(\mathbf{A}_m)\subseteq \mathbf{A}_n$, and set $T_1 = \{f\in H_n: f|_{\mathbf{A}_m}\in S\}$, $T_2 = \{f\in H_n: f\circ g|_{\mathbf{A}_m}\not\in S\}$. If $S\not\in \alpha g(m)$, then we have $T_1\cap T_2\in \alpha (n)$. Now pick $N$ large enough so that $\mathbf{A}_n\cup g(\mathbf{A}_n)\subseteq \mathbf{A}_N$. As $\alpha (n)$ is thick, fix $h\in H_N$ with $h\circ \emb{\mathbf{A}_n, \mathbf{A}_N}\subseteq T_1\cap T_2$. But now set $x = h\circ g|_{\mathbf{A}_n}\circ i^n_m = h\circ i^N_n\circ g|_{\mathbf{A}_m}$. Since $g|_{\mathbf{A}_n}\in \emb{\mathbf{A}_n, \mathbf{A}_N}$, we have $h\circ g|_{\mathbf{A}_n}\in T_1\cap T_2$, hence $x\in S$. Similarly $h\circ i^N_n \in T_1\cap T_2$, implying $x\not\in S$. This is a contradiction.

Conversely, if $\alpha(m)$ is not thick, suppose $S\in \alpha(m)$ is not thick, and find $n\geq m$ such that $f\circ \emb{\mathbf{A}_m,\mathbf{A}_n}\not\subseteq S$ for each $f\in H_n$. Then we have 
$$\bigcap_{r\in \emb{\mathbf{A}_m,\mathbf{A}_n}} \{f\in H_n: f \circ r\in S\} = \emptyset.$$
Hence for some $g\in G$, we must have $S\not\in \alpha g(m)$, and $\alpha$ cannot be a fixed point.
\end{proof}

\begin{rem}
M\"uller and Pongr\'acz in [MP] use different methods to show the following: let $\mathbf{K}$ be a \fr structure, $\mathcal{K} = \age{\mathbf{K}}$, and $G = \aut{\mathbf{K}}$. Suppose each $\mathbf{A}\in \mathcal{K}$ has Ramsey degree $\leq d$ for some fixed $d\in \mathbb{N}$. Then $|M(G)|\leq d$.
\end{rem}

\section{Metrizability of $M(G)$}

We now consider the case where $M(G)$ is metrizable. Corollary 6.4 tells us that if $M(G)$ is metrizable, then $M(G) = \varprojlim Y_n$, where $Y_n$ is a finite subset of $\beta H_n$. To characterize the ultrafilters that can appear in such a $Y_n$, we need to introduce some new terminology.

If $F_1,...,F_k$ are filters on $H_n$, we say that $\{F_1,...,F_k\}$ is \emph{thick} if every $S\in (F_1\cap\cdots\cap F_k)$ is thick.  It will often be the case that each $F_i$ is a filter on some $X_i\subseteq H_n$; when there is no confusion, we will identify $F_i$ with its pushforward to a filter on $H_n$. Note that if $\{F_1,...,F_k\}$ is thick and $F'$ is another filter on $H_n$, then $\{F_1,...,F_k, F'\}$ is also thick. We will frequently consider the following thick set of filters:

\begin{prop}
Let $\mathbf{A}_n$ have Ramsey degree $k$, and let $\gamma$ be a full syndetic $k$-coloring of $H_n$. Let $F_i\subseteq \mathcal{P}(\gamma_i)$ consist of those $X\subseteq \gamma_i$ which are syndetic. Then $\{F_1,...,F_k\}$ is a thick set of filters.
\end{prop}

\begin{proof}
First we show each $F_i$ is a filter; we prove this for $F_1$. Certainly $F_1$ is upward closed. Suppose $S,T\subseteq \gamma_1$ are syndetic. Form the $(k+3)$-coloring $\delta$ by letting $\delta_1 = (S\cap T)$, $\delta_{k+1} = S\setminus T$, $\delta_{k+2} = T\setminus S$, $\delta_{k+3} = \gamma_1\setminus (S\cup T)$, and $\delta_i = \gamma_i$ for $2\leq i\leq k$. If $(S\cap T)$ is not syndetic, then $H_n\setminus (S\cap T) = \delta_2\sqcup\cdots\sqcup \delta_{k+3}$ is thick. So some subset of $k$ colors among $\delta_2,...,\delta_{k+3}$ must form a thick subset. Since each $\delta_j$ for $2\leq j\leq k$ is syndetic, for one of $X = (S\setminus T), (T\setminus S), (\gamma_i\setminus (S\cup T))$ we have $\delta_2\sqcup\cdots\sqcup \delta_k\sqcup X$ thick. But this contradicts the fact that $S$ and $T$ are syndetic. Hence $S\cap T$ is syndetic, and $F_1$ is a filter. To see that $\{F_1,...,F_k\}$ is thick, pick $S_i\in F_i$ for $1\leq i\leq k$. Then  consider a full $2k$-coloring of $H_n$ with colors $S_i, (\gamma_i\setminus S_i)$ for $i\leq k$. Some $k$ equivalence classes form a thick subset; as each $S_i$ is syndetic, $S_i$ must be one of the equivalence classes.
\end{proof}

\begin{rem}
We will call $\{F_1,...,F_k\}$ as in Proposition 8.1 the \emph{syndetic filters for $\gamma$}.
\end{rem}

If $X\subseteq H_n$ is thick, we say that $S\subseteq H_n$ is \emph{syndetic relative to $X$} if $X\setminus S$ is not thick. Notice that if $Y\subseteq H_n$ is thick and $S$ is syndetic relative to $X$ for some $X\supseteq Y$, then $S$ is also syndetic relative to $Y$. We say that $S$ is \emph{piecewise syndetic} if $S$ is syndetic relative to some thick $X$. Equivalently, $S$ is piecewise syndetic if $S$ is the intersection of a syndetic set and a thick set. 

Now is a good time to compare our use of the words ``thick'', ``syndetic'', and ``piecewise syndetic'' with their traditional meanings (see [HS]). For $G$ a discrete group, $T\subseteq G$ is \emph{thick} if the collection $\{Tg^{-1}: g\in G\}$ has the finite intersection property. $S\subseteq G$ is \emph{syndetic} if there are $g_1,...,g_k\in G$ with $\bigcup_{i\leq k} Sg^{-1}_i = G$, and $P\subseteq G$ is \emph{piecewise syndetic} if there are $g_1,...,g_k$ with $\bigcup_{i\leq k} Pg^{-1}_i$ thick. It is not hard to show that $S$ is syndetic iff $G\setminus S$ is not thick, and $P$ is piecewise syndetic iff $P = S\cap T$ for $S$ syndetic and $T$ thick. 

Let us show that $T\subseteq H_n$ is thick in our sense iff $\pi_n^{-1}(T)$ is thick in the traditional sense. For the forward direction, fix $g_1,...,g_k\in G$. Find $N$ large enough so that $g_i(\mathbf{A}_n)\subseteq \mathbf{A}_N$ for each $i\leq k$. As $T$ is thick (in our sense), pick $x\in H_N$ with $x\circ \emb{\mathbf{A}_n,\mathbf{A}_N}\subseteq T$. Now let $g\in G$ be any element with $g|_{\mathbf{A}_N} = x$. Then $g\in \bigcap_{i\leq k} \pi_n^{-1}(T)g^{-1}_i$.

For the converse, if $T$ is not thick (in our sense), find $N$ large enough so that for any $x\in H_N$, $x\circ \emb{\mathbf{A}_n,\mathbf{A}_N}\not\subseteq T$. Find $g_1,...,g_k$ so that $\{g_i|_{\mathbf{A}_n}: i\leq k\} = \emb{\mathbf{A}_n,\mathbf{A}_N}$. Then $\bigcap_{i\leq k} \pi_n^{-1}(T)g^{-1}_i = \emptyset$. 

Now we see that $S\subseteq H_n$ is syndetic iff $\pi_n^{-1}(S)$ is syndetic, and similarly for piecewise syndetic. Let us return now to the story at hand.

\begin{prop}
The following are equivalent:
\begin{enumerate}
\item
$\mathbf{A}_n$ has Ramsey degree $t\leq k$,
\item
There is a thick set $\{\alpha_n^1,...,\alpha_n^k\}$ of $k$ ultrafilters on $H_n$. 
\end{enumerate}
\end{prop}
\begin{proof}\renewcommand{\qedsymbol}{} 
$(2\Rightarrow 1)$ Fix a $(k+1)$-coloring $\gamma$ of $H_n$. For each $1\leq i\leq k$, there is some $\gamma_{j_i}\in \alpha_n^i$. Then $\bigsqcup_{1\leq i\leq k} \gamma_{j_i}$ must be thick.

\vspace{2 mm}
$(1\Rightarrow 2)$ Fix $\gamma$ a syndetic $t$-coloring of $H_n$. Let $\{F_1,...,F_t\}$ be the syndetic filters for $\gamma$. We will be done once we prove the following lemma; we distinguish this lemma because it is somewhat stronger than what we need and we will use it later.
\end{proof}
\begin{lemma}
Suppose $\mathbf{A}_n$ has Ramsey degree $k$, $\gamma$ is a syndetic $k$-coloring, and $\{F_1,...,F_k\}$ are the syndetic filters for $\gamma$. Let $G_i$ be a filter on $\gamma_i$ extending $F_i$ such that $\{G_1,...,G_k\}$ is thick. Then each $G_i$ can be extended to an ultrafilter $U_i$ on $\gamma_i$ such that $\{U_1,...,U_k\}$ is thick.
\end{lemma}

\begin{proof}
We will show that $G_1$ can be extended to an ultrafilter $U_1$ such that $\{U_1, G_2,...,G_k\}$ is thick; by relabeling and repeating, this is enough. Let $P_1$ consist of those subsets $T\subseteq \gamma_1$ for which there are $S_i\in G_i$, $2\leq i\leq k$, such that $T$ is syndetic relative to $\gamma_1\sqcup S_2\sqcup\cdots \sqcup S_k$. I claim $P_1$ is a filter. Certainly $P_1$ is upward closed, so suppose $T_1, T_2\in P$. By taking intersections, we may suppose that there are $S_i\in \gamma_i$, $i\geq 2$, such that both $T_1$ and $T_2$ are syndetic relative to $\gamma_1\sqcup S_2\sqcup \cdots\sqcup S_k$. Now the proof that $T_1\cap T_2$ is syndetic relative to $\gamma_1\sqcup S_2\sqcup \cdots \sqcup S_k$ mimics the proof of Proposition 8.1. 
\vspace{1.5 mm}

Now let $S\in G_1$, and suppose $T\in P_1$ as witnessed by $S_i\in G_i$, $2\leq i\leq k$. Then I claim $(S\cap T)\cup S_2\cup \cdots\cup S_k$ is thick. Consider the $(k+1)$-coloring $\delta$ with $\dom{\delta} = S\sqcup S_2\sqcup\cdots \sqcup S_k$ and with $\delta_1 = (S\cap T)$, $\delta_{k+1} = (S\setminus T)$, and $\delta_i = S_i$ for $2\leq i\leq k$. $\delta$ is large, and we cannot have $(S\setminus T)\cup S_1\cup\cdots\cup S_k$ thick since $T$ is syndetic relative to $\gamma_1\sqcup S_2\sqcup\cdots \sqcup S_k$. So as each $\gamma_i$ is syndetic, we must have $(S\cap T)\cup S_2\cup\cdots\cup S_k$ thick. In particular, since $\gamma_1$ is syndetic, $S\cap T$ is non-empty.

We can now extend the filter generated by $P_1$ and $G_1$ to an ultrafilter $U_1$. Since this ultrafilter extends $P_1$, $\{U_1,G_2,...,G_k\}$ is thick. 
\end{proof}

We need to develop a few ideas related to colorings before proceeding. If $\gamma$ is a $k$-coloring and $\delta$ is an $\ell$-coloring both with domain $X$, the \emph{product coloring} $\gamma*\delta$ is the $k\ell$-coloring with domain $X$ with $\gamma*\delta(x) = \gamma(x)(\ell-1) + \delta(x)$. We say that $\delta$ \emph{refines} $\gamma$ if $\delta(x) = \delta(y)$ implies $\gamma(x) = \gamma(y)$. If $\gamma$ is a coloring of $H_m$ and $f\in \emb{\mathbf{A}_m,\mathbf{A}_n}$, then $f(\gamma)$ is the coloring of $H_n$ with $\mathrm{dom}(f(\gamma)) = \hat{f}^{-1}(\mathrm{dom}(\gamma))$ and $f(\gamma)(x) = \gamma(x\circ f)$.
\vspace{1.5 mm}

$G$ acts on $H_n$ via $g\cdot x = g\circ x$. This induces a continuous (left) logic action on the compact, metrizable space of partial $k$-colorings with at most $k$ colors. Explicitly, $g\gamma(g\cdot x)$ is defined iff $\gamma(x)$ is, and $g\gamma(g\cdot x) = \gamma(x)$. This will be the only time we use left actions in this paper. Below we collect some simple facts about colorings.

\begin{enumerate}
\item
If $\gamma$ is a syndetic coloring of $H_m$ and $f\in \emb{\mathbf{A}_m,\mathbf{A}_n}$, then $f(\gamma)$ is a syndetic coloring of $H_n$.
\item
If $\mathbf{A}_n$ has Ramsey degree $k$, then for every large $\ell$-coloring $\gamma$ of $H_n$ with $k \leq \ell$, there is (up to relabeling colors) a full $k$-coloring $\gamma'\in \overline{G\cdot \gamma}$.
\item
If $\gamma$ is a full syndetic $k$-coloring of $H_n$, then every $\gamma'\in \overline{G\cdot \gamma}$ is a full syndetic $k$-coloring.
\item
Let $\gamma, \delta$ be full colorings of $H_n$ such that $\delta$ refines $\gamma$. If $g_N\cdot \delta\rightarrow \delta'$, then $g_N\cdot \gamma$ also converges to some $\gamma'$, and $\delta'$ refines $\gamma'$.
\end{enumerate}

\begin{lemma}
Suppose $m\leq n$, and $\mathbf{A}_m$ and $\mathbf{A}_n$ have Ramsey degrees $k$ and $\ell$, respectively, with $k\leq \ell$. Then there are syndetic colorings $\gamma, \delta$ of $H_m, H_n$ in $k, \ell$ colors, respectively, such that $\delta$ refines $i^n_m(\gamma)$.
\end{lemma}

\begin{proof}
Choose any full syndetic $k,\ell$ colorings $\gamma', \delta'$ of $H_m, H_n$, respectively. Form the product coloring $P:= i^n_m(\gamma')*\delta'$ on $H_n$. $\overline{G\cdot P}$ must contain a full $\ell$-coloring $\delta$ (up to relabeling colors) which must also be syndetic. If $g_N\cdot P\rightarrow \delta$, then $g_N\cdot \gamma'$ converges to some coloring $\gamma$. $\gamma$ and $\delta$ are as desired.
\end{proof}

 If $F$ is a filter on $H_m$ and $f\in\emb{\mathbf{A}_m,\mathbf{A}_n}$, introduce the shorthand notation $f(F)$ for $\hat{f}^{-1}(F)$. Notice that in the proof of Proposition 7.2, we showed that if $X\subseteq H_m$ is thick, then $\hat{f}^{-1}(X)\subseteq H_n$ is also thick; it follows that if $F_1,...,F_k$ are filters on $H_m$ and $\{F_1,...,F_k\}$ is thick, then $\{f(F_1),...,f(F_k)\}$ is also thick. The following proposition is similar in spirit to Proposition 7.2.

\begin{prop}
Suppose $m\leq n$ and $\mathbf{A}_m$, $\mathbf{A}_n$ have Ramsey degrees $k\leq \ell$, respectively. Then if $\{\alpha_m^i: 1\leq i\leq k\}\subseteq \beta H_m$ is thick, there is a thick set $\{\alpha_n^j: 1\leq j\leq \ell\}\subseteq H_n$ such that $\{\tilde{i}^n_m(\alpha_n^j): 1\leq j\leq \ell\} = \{\alpha_m^i: 1\leq i\leq k\}$.
\end{prop}

\begin{proof}
Fix full syndetic $k, \ell$ colorings $\gamma, \delta$ of $H_m, H_n$, respectively, with $\delta$ refining $i_m^n(\gamma)$ as guaranteed by Lemma 8.4. For $1\leq j\leq \ell$, let $a_j$ be the unique number $1\leq a_j\leq k$ with $\delta_j\subseteq i_m^n(\gamma)_{a_j}$. Notice that since $\{\alpha_m^1,...,\alpha_m^k\}$ is thick and since each $\gamma_i$ is syndetic, without loss of generality we may suppose $\gamma_i\in\alpha_m^i$. We can then assume that $\alpha_m^i$ is an ultrafilter on $\gamma_i$.

Let $P_j$ be the filter on $\delta_j$ with $T\in P_j$ iff $T\supseteq \delta_j\cap S$ for some $S\in i^n_m(\alpha_m^{a_j})$; as $\delta_j$ is syndetic and $\{i_m^n(\alpha_m^1),...,i_m^n(\alpha_m^k)\}$ is thick, $\delta_j\cap S \neq \emptyset$ for any $S\in i_m^n(\alpha_m^{a_j})$, so $P_j$ is indeed a filter. Let us show that $\{P_1,...,P_\ell\}$ is thick. Pick $T_j\in P_j$ as witnessed by $S_j\in i_m^n(\alpha_m^{a_j})$ for each $1\leq j\leq \ell$. By taking intersections, we may assume $S_j$ depends only on $a_j$. It follows that $S_j\subseteq T_1\cup\cdots\cup T_\ell$. As $j\rightarrow a_j$ is onto and $\{i_m^n(\alpha_m^1),...,i_m^n(\alpha_m^k)\}$ is thick, we are done.
\vspace{1.5 mm}

Let $\{F_1,...,F_\ell\}$ be the syndetic filters for $\delta$. If $T_i'\in F_i$ and $T_i\in P_i$ for each $1\leq i\leq \ell$, then I claim that $(T_1'\cap T_1)\cup\cdots\cup (T'_\ell\cap T_\ell)$ is thick. Consider the $2\ell$-coloring of $T_1\cup\cdots\cup T_\ell$ with colors $(T_j'\cap T_j)$ and $(T_j\setminus T_j')$ for $1\leq j\leq \ell$. Some $\ell$ colors must form a thick subset, and each $T_j'$ is syndetic.

Therefore let $G_j$ be the filter generated by $F_j\cup P_j$; use Lemma 8.3 to obtain a thick set of ultrafilters $\{\alpha_n^j: 1\leq j\leq \ell\}$. Since $\alpha_n^j$ extends $P_j$, we have that $\tilde{i}_m^n(\alpha_n^j) = \alpha_m^{a_j}$, completing the proof. 
\end{proof}

\begin{cor}
If each $\mathbf{A}\in\mathcal{K}$ has finite Ramsey degree, there is $\varprojlim Y_n \subseteq \varprojlim \beta H_n$ with $Y_n\subseteq \beta H_n$ a finite thick set. 
\end{cor}

\begin{theorem}
Let $\mathbf{K}$ be a \fr structure, with $\mathcal{K} = \age{\mathbf{K}}$ and $G = \aut{\mathbf{K}}$. Then the following are equivalent:
\begin{enumerate}
\item
$M(G)$ is metrizable,
\item
Each $\mathbf{A}\in \mathcal{K}$ has finite Ramsey degree.
\end{enumerate}
\end{theorem}

\begin{proof}
Suppose $Y = \varprojlim Y_n\subseteq \varprojlim \beta H_n$ with $Y_n\subseteq H_n$ finite. We will show that there is $Z\subseteq Y$ with $Z$ a subflow of $\varprojlim \beta H_n$ iff for each $n$, $Y_n$ is thick. Set $Y_n = \{\alpha_n^i: 1\leq i\leq k_n\}$. Let $F_n$ be the filter $(\alpha_n^1\cap\cdots\cap \alpha_n^{k_n})$.

Suppose $S\in F_m$ is not thick. Pick $\alpha\in Y$. The proof that there is $g\in G$ with $S\not\in \alpha g(m)$ now proceeds exactly the same as the second paragraph of the proof of Theorem 7.3.

Now suppose for each $n$ that $Y_n$ is thick. For $W\subseteq G$ finite, $m\in\mathbb{N}$, and $S\in F_m$, let $Y_{W, S}$ consist of those $\alpha\in Y$ such that $S\in \alpha g(m)$ for each $g\in W$. Notice that $Y_{W,S}\subseteq Y$ is closed, hence compact.
\vspace{2 mm} 

\begin{claim} 
First, let us show that $Y_{W, S}$ is nonempty. Fix $n$ large enough so that $g(\mathbf{A}_m)\subseteq \mathbf{A}_n$ for each $g\in W\cup \{1_G\}$. For $g\in G$, set $T_g = \{f\in H_n: f\circ g|_{\mathbf{A}_m}\in S\}$. We will show that the set $X := T_{1_G}\setminus \left(\bigcap_{g\in W} T_g\right)$ is not thick by mimicking the first paragraph of the proof of Theorem 7.3. If it were, pick $N$ large enough so that $g(\mathbf{A}_n)\subseteq \mathbf{A}_N$ for each $g\in W\cup \{1_G\}$ and find $h\in H_N$ so that $h\circ \emb{\mathbf{A}_n,\mathbf{A}_N}\subseteq X$. But now for each $g\in W$, set $x_g = h\circ g|_{\mathbf{A}_n}\circ i_m^n = h\circ i^N_n\circ g|_{\mathbf{A}_m}$. Since $g|_{\mathbf{A}_n}$ and $i^N_n$ are both in $\emb{\mathbf{A}_n,\mathbf{A}_N}$, we have $h\circ g|_{\mathbf{A}_n}$ and $h\circ i^N_n$ in $X$. Since $h\circ g|_{\mathbf{A}_n}\in X\subseteq T_{1_G}$, we have $x_g\in S$. But this implies that $h\circ i^N_n\in\bigcap_{g\in W} T_g$, a contradiction.

Since $Y_n$ is thick and since $T_{1_G} = (\hat{i}_m^n)^{-1}(S)\in Y_n$, this means that $(\bigcap_{g\in W\cup \{1_G\}} T_g) \in \alpha_n^j$ for some $\alpha_n^j\in Y_n$. Now any $\alpha\in Y$ with $\alpha (n) = \alpha_n^j$ is a member of $Y_{W,S}$. This proves the claim.
\end{claim}
\vspace{2 mm}

Now observe that if $W_1,W_2$ are finite subsets of $G$, $S_1\in F_m$, and $S_2\in F_n$ ($m\leq n$), then letting $S_3 = (\hat{i}_m^n)^{-1}(S_1)\cap S_2\in F_n$, we have $Y_{W_1,S_1}\cap Y_{W_2, S_2} \supseteq Y_{W_1\cup W_2, S_3}$. In particular, since each $Y_{W,S}$ is compact, there is $\alpha\in Y$ a member of all of them. Hence $\overline{\alpha\cdot G}\subseteq Y$ is a metrizable subflow of $\varprojlim \beta H_n$. 
\end{proof}

\begin{cor}
If $Y = \varprojlim Y_n\subseteq \varprojlim \beta H_n$, each $Y_n$ is thick, and $\mathbf{A}_n$ has Ramsey degree $|Y_n|<\infty$, then $Y\cong M(G)$.
\end{cor}
\vspace{2 mm}

Suppose $Y = \varprojlim Y_n\subseteq \varprojlim \beta H_n$ is the universal minimal flow and is metrizable. By Corollary 8.8, we may assume that $|Y_n| := k_n$ is the Ramsey degree of $\mathbf{A}_n$. It will be useful now to abuse notation and think of $\mathcal{K}$ as being the Fra\"iss\'e--HP class $\{\mathbf{B}: \mathbf{B}\cong \mathbf{A}_n \text{ for some $n$}\}$. Our goal is to interpret $\bigcup_{n\in \mathbb{N}}Y_n$ as $\mathrm{Cat}_X(\mathcal{K}(Y), \mathcal{K})$ for some adequate $X$ and expansion class $\mathcal{K}(Y)$ so that $(\mathcal{K}(Y),\mathcal{K})$ is excellent.

If $f\in \emb{\mathbf{A}_m, \mathbf{A}_n}$, $\alpha_m^i\in Y_m$, $\alpha_n^j\in Y_n$, then we say that $f\in \emb{\alpha_m^i, \alpha_n^j}$ if for all $g\in G$ with $g|_{\mathbf{A}_m} = f$ and all $\alpha\in Y$ with $\alpha (n) = \alpha_n^j$, we have $\alpha g (m) = \alpha_m^i$. If $f\in H_m$, we say that $f\in \emb{\alpha_m^i, \alpha}$ if $f\in \emb{\alpha_m^i, \alpha (n)}$ for large enough $n$.

\begin{prop}
If for some $g\in G$ with $g|_{\mathbf{A}_m} = f$ and some $\alpha\in Y$ with $\alpha (n) = \alpha_n^j$ we have $\alpha g(m) = \alpha_m^i$, then $f\in \emb{\alpha_m^i, \alpha_n^j}$. If $f_1\in \emb{\alpha_m^{i_1},\alpha_n^{i_2}}$ and $f_2\in \emb{\alpha_n^{i_2}, \alpha_N^{i_3}}$, then $f_2\circ f_1\in \emb{\alpha_m^{i_1}, \alpha_N^{i_3}}$. 
\end{prop}

\begin{proof}
The second statement easily follows from the first. Now suppose $g$, $\alpha$ are as above; fix $S\subseteq H_m$. Let $T = \{f\in H_n: f\circ g|_{\mathbf{A}_m}\in S\}$. Then $S\in \alpha g(m)$ iff $T\in \alpha (n)$. In particular, $\alpha g(m)$ depends only on $g|_{\mathbf{A}_m}$ and $\alpha (n)$.
\end{proof}

Define $C$ to be the category with object set $\bigcup_{n\in \mathbb{N}} Y_n$ and arrows defined as above. To realize $C$ as $\mathrm{Cat}_{\{\mathbf{A}_n: n\in \mathbb{N}\}}(\mathcal{K}(Y), \mathcal{K})$ for some expansion $\mathcal{K}(Y)$, fix an enumeration $a_n^1,...,a_n^{N_n}$ of the elements of each $\mathbf{A}_n$. For each $\alpha_m^i\in Y_m$, introduce a new $N_m$-ary relation symbol $R_m^i$. Then the object $\alpha_n^j$ can now be realized as the structure $\langle \mathbf{A}_n, (R_m^i)_{m\in \mathbb{N}, 1\leq i\leq k_m}\rangle$, where $R_m^i(b_1,...,b_{N_m})$ holds whenever the map $a_m^i\rightarrow b_i$ is in $\emb{\alpha_m^i,\alpha_n^j}$. If $\alpha\in Y$, we can interpret $\alpha$ as the countably infinite locally finite structure $\bigcup_n \alpha(n)$. Notice that $Y\cong X_{\mathcal{K}(Y)}$.

\begin{prop}
$\mathcal{K}(Y)$ is a reasonable precompact expansion of $\mathcal{K}$ which has the JEP and the ExpP.
\end{prop}

\begin{proof}
If $\alpha_m^i,\alpha_n^j\in \mathcal{K}(Y)$, pick $\alpha\in Y$ with $\alpha (m) = \alpha_m^i$. As $Y$ is minimal, find $g\in G$ with $\alpha g(n) = \alpha_n^j$. Let $N$ be large enough so that $g(\mathbf{A}_n)\subseteq \mathbf{A}_N$. Then $\alpha (N)$ witnesses JEP for $\alpha_m^i$, $\alpha_n^j$ as witnessed by the maps $i^N_m$ and $g|_{\mathbf{A}_n}$. 

Now let $f\in \emb{\mathbf{A}_m,\mathbf{A}_n}$ and $\alpha_m^i\in Y_m$. Pick $g\in G$ with $g|_{\mathbf{A}_m} = f$, and pick $\alpha\in Y$ with $\alpha (m) = \alpha_m^i$. Then $\alpha g^{-1}(n)$ is an expansion of $\mathbf{A}_n$ with $f\in \emb{\alpha_m^i,\alpha_n^j}$, showing that $\mathcal{K}(Y)$ is a reasonable expansion of $\mathcal{K}$. 

Suppose $\mathcal{K}(Y)$ did not have the ExpP as witnessed by $\alpha_m^i$. For each $N\in\mathbb{N}$ pick $\alpha_N\in Y$ with $\emb{\alpha_m^i, \alpha_N(N)} = \emptyset$. By passing to a convergent subsequence, suppose $\alpha_N\rightarrow \alpha$. Then $\overline{\alpha\cdot G}\cap \{\zeta\in Y: \zeta(m) = \alpha_m^i\} = \emptyset$, a contradiction to minimality of $Y$.
\end{proof}

Notice that the ExpP implies the following useful corollary:

\begin{cor}
For any $\alpha\in Y$ and any $\alpha_m^i\in \mathcal{K}(Y)$, $\emb{\alpha_m^i, \alpha}$ is a syndetic subset of $H_m$.
\end{cor}

\begin{prop}
$\mathcal{K}(Y)$ has the Ramsey Property.
\end{prop}

\begin{proof}
Pick any $\alpha\in Y$, and fix $\alpha_n^1\in \mathcal{K}(Y)$. By Proposition 4.1, it is enough to show that for any full 2-coloring $\gamma$ of the set $\emb{\alpha_n^1, \alpha}$, there is a color $\gamma_j$ which is a thick subset of $\emb{\alpha_n^1, \alpha}$. This is equivalent to showing that $\gamma_j\cup \left(\bigcup_{2\leq i \leq k_n} \emb{\alpha_n^i, \alpha}\right)$ is a thick subset of $H_n$. But consider the $(k_n+1)$-coloring $\delta$ of $H_n$ given by letting $\delta_1 = \gamma_1$, $\delta_{k_n+1} = \gamma_2$, and $\delta_i = \emb{\alpha_i^n, \alpha}$ for $2\leq i \leq k_n$; since $\mathbf{A}_n$ has Ramsey degree $k_n$ and by Corollary 8.11, we are done.
\end{proof}

To show that $\mathcal{K}(Y)$ has the AP, we note the following theorem of Ne\v{s}et\v{r}il and R\"odl [NR].
\begin{prop}
Let $\mathcal{C}$ be a class of finite structures with the JEP and the RP. Then $\mathcal{C}$ has the AP.
\end{prop}

\begin{proof}
Fix $\mathbf{A},\mathbf{B},\mathbf{C}\in\mathcal{C}$ and embeddings $f:\mathbf{A}\rightarrow \mathbf{B}$ and $g: \mathbf{A}\rightarrow \mathbf{C}$. Let $\mathbf{E}$ witness JEP for $\mathbf{B}$ and $\mathbf{C}$, and find $\mathbf{D}\in\mathbf{C}$ with $\mathbf{D}\rightarrow (\mathbf{E})^{\mathbf{A}}_4$.

Let $S:= \{ x\in \emb{\mathbf{A},\mathbf{D}}: \exists y\in\emb{\mathbf{B},\mathbf{D}}(x = y\circ f)\}$. Let $T$ be defined similarly for $\mathbf{C}$ and $g$. Let $\gamma$ be the $4$-coloring of $\emb{\mathbf{A},\mathbf{D}}$ with $\gamma_1 = S\cap T$, $\gamma_2 = S\setminus T$, $\gamma_3 = T\setminus S$, and $\gamma_4 = \emb{\mathbf{A},\mathbf{D}}\setminus (S\cup T)$. Pick $h\in\emb{\mathbf{E},\mathbf{D}}$ with $h\circ \emb{\mathbf{A},\mathbf{E}}$ monochromatic. Since $\mathbf{E}$ embeds both $\mathbf{B}$ and $\mathbf{C}$, the color must be $\gamma_1$. Now pick any $x\in \emb{\mathbf{A},\mathbf{E}}$, and pick $r, s$ witnessing the fact that $h\circ x\in S\cap T$. Then $h\circ x = r\circ f = s\circ g$.
\end{proof} 

We now have the following:

\begin{theorem}
Let $\mathbf{K}$ be a \fr structure with $\mathcal{K} = \age{\mathbf{K}}$ and $G = \aut{\mathbf{K}}$. Then the following are equivalent:
\begin{enumerate}
\item
$G$ has metrizable universal minimal flow,
\item 
Each $\mathbf{A}\in \mathcal{K}$ has finite Ramsey degree,
\item
There is a \fr precompact expansion class $\mathcal{K}^*$ with $(\mathcal{K}^*,\mathcal{K})$ excellent, i.e.\ $\mathcal{K}^*$ is reasonable and has both the ExpP and the RP.
\end{enumerate}
\end{theorem}

\begin{proof}
$(1\Leftrightarrow 2)$ was Theorem 8.7, and we have just shown $(2\Rightarrow 3)$. $(3\Rightarrow 2)$ is Proposition 5.8. 
\end{proof}

In light of the discussion before Proposition 5.10, we obtain:

\begin{cor}
If $G$ is a closed subgroup of $S_\infty$ with metrizable universal minimal flow, then $G$ has the generic point property.
\end{cor}
\vspace{2 mm}

We conclude with the proof of Theorem 5.7. Suppose $\mathcal{K}^*$ is a \fr precompact expansion class as in Theorem 8.14 (3). We will show that $X_{\mathcal{K}^*}\cong M(G)$. Let $\mathbf{K}^*\in X_{\mathcal{K}^*}$ be the \fr limit. For each $n$, let $\mathbf{A}_n^1,...,\mathbf{A}_n^{k_n}$ be the expansions of $\mathbf{A}_n$ in $\mathcal{K}^*$, with $\mathbf{K}^*|_{\mathbf{A}_n} = \mathbf{A}_n^1$. Write $H_n^*$ for $\emb{\mathbf{A}_n^1, \mathbf{K}^*}$. Notice that $\varprojlim \beta H_n^*\subseteq \varprojlim \beta H_n$ is nonempty. Pick $\alpha\in \varprojlim \beta H_n^*$ so that each $\alpha(n)$ is thick for $H_n^*$ (we can do this as $\mathcal{K}^*$ has RP).

If $f, x\in H_m$, let us say that $f$ and $x$ have the \emph{same expansion} if for large enough $n$, there is $r\in H_n^*$ with $r\circ f = x$. It is straightforward to check that this is an equivalence relation. 
\vspace{2 mm}

\begin{claim} For $g,h\in G$, $g|_{\mathbf{A}_m}$ and $h|_{\mathbf{A}_m}$ have the same expansion iff $\alpha g(m) = \alpha h(m)$; we once again mimic the proof of Theorem 7.3. Suppose $g, h\in G$ are such that $g|_{\mathbf{A}_m}$ and $h|_{\mathbf{A}_m}$ have the same expansion. Pick $n$ large enough so that $g(\mathbf{A}_m)\cup h(\mathbf{A}_m)\subseteq \mathbf{A}_n$. Let $S\in \alpha g(m)$; set $T_g = \{f\in H_n: f\circ g|_{\mathbf{A}_m}\in S\}$ and $T_h = \{f\in H_n: f\circ h|_{\mathbf{A}_m}\in S\}$. For sake of contradiction, suppose $S\not\in \alpha h(m)$, so that $T_g\setminus T_h\in \alpha (n)$. Since $g|_{\mathbf{A}_m}$ and $h|_{\mathbf{A}_m}$ have the same expansion, find $r\in H_n^*$ with $r\circ g|_{\mathbf{A}_m} = h|_{\mathbf{A}_m}$. Pick $N$ large enough so that $\mathbf{A}_n\cup r(\mathbf{A}_n)\subseteq \mathbf{A}_N$. Since $\alpha(n)$ is thick for $H_n^*$, find $s\in H_N^*$ so that $s\circ \emb{\mathbf{A}_n^1, \mathbf{A}_N^1}\subseteq T_g\setminus T_h$. But now set $x = s\circ r\circ g|_{\mathbf{A}_m} = s\circ i^N_n\circ h|_{\mathbf{A}_m}$. This tells us that $x\in S$ and $x\not\in S$, a contradiction.

For the converse, suppose $\alpha g(m) = \alpha h(m)$. Pick any $S\in \alpha g(m)$ with $S\subseteq H_m^*$. Then with $T_g, T_h$ as above, we have $T_g\cap T_h\cap H_n^*\in \alpha(n)$. Pick $f\in T_g\cap T_h\cap H_n^*$. Then $f\circ g|_{\mathbf{A}_m}$ and $f\circ h|_{\mathbf{A}_m}$ are in $H_m^*$; since $\mathbf{K}^*$ is ultrahomogeneous, $f\circ g|_{\mathbf{A}_m}$ and $f\circ h|_{\mathbf{A}_m}$ have the same expansion. It follows that $g|_{\mathbf{A}_m}$ and $h|_{\mathbf{A}_m}$ also have the same expansion. This proves the claim.
\end{claim}
\vspace{2 mm}

It now follows that $\overline{\alpha \cdot G} \subseteq \varprojlim Y_n := Y$ with $|Y_n|\leq k_n$. By the proof of Theorem 8.7, we see that each $Y_n$ is thick. Since the Ramsey degree of $\mathbf{A}_n$ is exactly $k_n$ by Proposition 5.8, we must have $|Y_n| = k_n$; hence by Corollary 8.8, we have $\overline{\alpha\cdot G} = Y \cong M(G)$. Form the expansion class $\mathcal{K}(Y)$. Note that each structure in $\mathcal{K}(Y)$ is isomorphic to $\alpha g(m)$ for some $g\in G$ and $m\in \mathbb{N}$.  Pick $f\in H_m$ and $g\in G$ with $g|_{\mathbf{A}_m} = f$. Then for $h_1, h_2\in G$, we have
\begin{align*}
f\in \emb{\alpha h_1(m), \alpha h_2(n)} &\Leftrightarrow \alpha h_2g(m) = \alpha h_1(m) \\
&\Leftrightarrow h_2\circ f \text{ and } h_1|_{\mathbf{A}_m} \text{ have the same expansion} \\
&\Leftrightarrow f\in \emb{\mathbf{A}_m(h_1|_{\mathbf{A}_m}, \mathbf{K}^*), \mathbf{A}_n(h_2|_{\mathbf{A}_n}, \mathbf{K}^*)}
\end{align*}
Letting $X = \{\mathbf{A}_n: n\in \mathbb{N}\}$, this shows that $\Phi_X: \mathrm{Cat}_X(\mathcal{K}(Y), \mathcal{K})\rightarrow \mathrm{Cat}_X(\mathcal{K}^*, \mathcal{K})$ given by $\Phi_X(\alpha g(m)) = \mathbf{A}_m(g|_{\mathbf{A}_m}, \mathbf{K}^*)$ is an isomorphism of expansions. Hence $X_{\mathcal{K}^*}\cong M(G)$. 

\section{Conclusion}

While the new proof of KPT correspondence given here solves the Generic Point Problem for closed subgroups of $S_\infty$, it was originally stated for any Polish group $G$. We briefly discuss one possible generalization of the methods presented here. 

A (relational) \emph{metric structure} is of the form $\langle X, d, \{R_i: i\in I\}\rangle$, where $X$ is a Polish metric space, $d$ is the metric (we assume that $X$ has diameter less than 1), and the ``relations'' $R_i: X^{n_i}\rightarrow \mathbb{R}$ are $n_i$-ary functions which are $k$-Lipschitz for some $k$. An automorphism of the structure is then an isometry of $(X,d)$ which in addition preserves all of the relations $R_i$. The \emph{quantifier-free type} of a finite tuple $(x_1,...,x_k)$ is just the (labelled) induced substructure on $\{x_1,...,x_k\}$. In particular, $(x_1,...,x_k)$ and $(y_1,...,y_k)$ have the same quantifier-free type iff $x_i\rightarrow y_i$ is an isomorphism of the induced substructures. 

A metric structure $\mathbf{X}$ is said to be \emph{near-ultrahomogeneous} if for any $(x_1,...,x_k), (y_1,...,y_k)$ with the same quantifier-free type and any $\epsilon >0$, there is an automorphism $\pi$ of $\mathbf{X}$ with $\mathrm{max}_i(d(\pi(x_i), y_i)) < \epsilon$. Near-ultrahomogeneous metric structures are called \emph{metric \fr structures}. One of the main theorems of metric \fr theory is that for any Polish group G, there is a metric \fr structure $\mathbf{X}$ with $\aut{\mathbf{X}}\cong G$; here $\aut{\mathbf{X}}$ is given the pointwise convergence topology. One can also consider the \emph{metric \fr class} $\mathcal{X}$ of finite structures which embed into $\mathbf{X}$. By no means is this intended to be a complete introduction to the theory; the interested reader should see [MT] and [Sch].

There is evidence that metric \fr theory can be used to investigate the dynamical properties of Polish groups. Melleray and Tsankov in [MT] have shown that $\aut{\mathbf{X}}$ is extremely amenable iff the class $\mathcal{X}$ satisfies an appropriate analogue of the Ramsey Property. Perhaps it is possible to use methods similar to those in section 6 to characterize the greatest ambit. We might proceed as follows: enumerate $D = \{x_1,x_2,...\}$ a dense subset of $X$, and let $\mathbf{X}_n$ be the induced substructure on $\{x_1,...,x_n\}$. Let $i_m^n$ be the inclusion of $\mathbf{X}_m$ into $\mathbf{X}_n$ for $m\leq n$ as before. We can, as before, set $H_n$ to be the set of embeddings of $\mathbf{X}_n$ into $\mathbf{X}$. Here a major difference arises; now $H_n$ has the structure of a Polish metric space rather than a discrete structure. However, any metric space (in particular, any Tychonoff space) admits a $\beta$-compactification (if the space is not locally compact, the original space may not be an open subspace of the compactification). Since the maps $\hat{i}_m^n: H_n\rightarrow H_m$ have dense image, the unique extension 
$\tilde{i}_m^n$ is surjective, and we can form $\varprojlim \beta H_n$ as before. Defining a jointly continuous $G$-action and showing that this is the greatest ambit seem to require a bit more care. 

Some of the results in this paper are known to generalize to general Polish groups. Melleray, Nguyen Van Th\'e, and Tsankov [MNT] have shown that if $G$ is a Polish group, $M(G)$ metrizable, and $G$ has the generic point property, then $M(G)$ is of the form $\widehat{G/G_0}$, where $G_0$ is extremely amenable and coprecompact in $G$, and the completion is taken with respect to the left uniformity on $G/G_0$ (using the left uniformity yields a right $G$-action). In particular, for $G$ a closed subgroup of $S_\infty$, saying that $M(G)$ is of the form $\widehat{G/G_0}$ for $G_0$ extremely amenable and coprecompact is exactly to say that 8.14 (3) holds.

We conclude with two problems. In my opinion, the second is much more difficult and fundamental.

\begin{prob}
Use metric \fr theory and methods similar to section 6 to characterize the greatest $G$-ambit for any Polish group $G$.
\end{prob}

Presuming that Problem 9.1 can be solved, we then would like to use the greatest ambit to characterize when $M(G)$ is metrizable. However, we need an analogue of Theorem 2.9 for $\beta$-compactifications of arbitrary Polish spaces.

\begin{prob}
Let $X$ be a Polish space, and form $\beta X$. Characterize the compact metric spaces which embed into $\beta X$.
\end{prob}

Andy Zucker

Carnegie Mellon University

Pittsburgh, PA 15213

zucker.andy@gmail.com
\end{document}